\documentclass[leqno,twoside,11pt]{amsart}

\usepackage{latexsym,esint, dsfont}

\usepackage{color}

\theoremstyle{plain}

\newtheorem{theorem}{Theorem}[section]
\newtheorem{corollary}[theorem]{Corollary}
\newtheorem{lemma}[theorem]{Lemma}

\theoremstyle{definition}

\numberwithin{equation}{section}

\DeclareMathOperator{\dist}{dist}
\DeclareMathOperator{\supp}{supp}
\DeclareMathOperator{\sym}{sym}
\DeclareMathOperator{\proj}{\mathbb{P}_{SO(3)}}

\newcommand{\abs}[1]{\left\lvert #1 \right\rvert}
\newcommand{\norm}[1]{\left\lVert #1 \right\rVert}
\newcommand{\inn}[2]{\langle #1\,,\,#2 \rangle}
\newcommand{\pare}[1]{\left(#1 \right)}
\newcommand{\bigo}[1]{\mathcal{O}(#1)}
\newcommand{\smo}[1]{o(#1)}
\newcommand{\car}[1]{\mathds{1}_{#1}}

\newcommand{\Vh}{\mathcal{V}^h}
\newcommand{\Uh}{\mathcal{U}^h}

\newcommand{\nabu}{\nabla{u^h}}

\newcommand{\Dhx}{D_{x',h}}
\newcommand{\Bhx}{B_{x',h}}
\newcommand{\Bhz}{B_{z',h}}
\newcommand{\Rhx}{R_{x',2h}}
\newcommand{\Ex}{\eta_{x'}}
\newcommand{\Exz}{\eta_{x'}(z')}
\newcommand{\Eyz}{\eta_{z'}(y')}
\newcommand{\dx}{\,\mathrm{d}x}
\newcommand{\dy}{\,\mathrm{d}y}
\newcommand{\dz}{\,\mathrm{d}z}

\newcommand{\dxt}{\,\mathrm{d}x_3}
\newcommand{\Id}{\mbox{Id}_3}

\renewcommand{\b}{\vec{b}_0}
\renewcommand{\d}{\vec{d}_0}
\renewcommand{\k}{\vec{k}_0}
\newcommand{\p}{\vec{p}}
\newcommand{\ph}{\vec{p}^h}
\newcommand{\qh}{\vec{q}^h}
\newcommand{\rh}{\vec{r}^h}

\begin{document}
\title[incompatible prestrain of higher order]
{Plates with incompatible prestrain of higher order}
\author{Marta Lewicka}
\address{Marta Lewicka,  University of Pittsburgh, Department of Mathematics, 
301 Thackeray Hall, Pittsburgh, PA 15260, USA }
\email{lewicka@pitt.edu}
\author{Annie Raoult}
\address{Annie Raoult, Laboratoire MAP5, 
  Universit\'e Paris Descartes \& CNRS, Sorbonne Paris Cit\'e, France}
\email{annie.raoult@parisdescartes.fr} 
\author{Diego Ricciotti}
\address{Diego Ricciotti, University of Pittsburgh, Department of Mathematics, 
301 Thackeray Hall, Pittsburgh, PA 15260, USA}
\email{DIR17@pitt.edu}

\date{23 March, 2015}

\begin{abstract}
We study the effective elastic behaviour of the incompatibly prestrained
thin plates, characterized by a Riemann metric $G$ on the
reference configuration. We assume that the prestrain is ``weak'', i.e. it induces scaling
of the incompatible elastic energy $E^h$ of order less than $h^2$ in terms
of the plate's thickness $h$. 

We essentially prove two results. First, we establish the
$\Gamma$-limit of the scaled energies $h^{-4}E^h$ and show that it
consists of a von K\'arm\'an-like energy, given in terms of the first
order infinitesimal isometries and of the admissible
strains on the surface isometrically immersing $G_{2\times 2}$ (i.e. the prestrain metric on the
midplate) in $\mathbb{R}^3$. Second, we prove that in the scaling regime $E^h\sim h^\beta$ with
$\beta>2$, there is no other limiting theory: if $\inf h^{-2} E^h \to
0$ then $\inf E^h\leq Ch^4$, and if $\inf h^{-4}E^h\to 0$ then $G$ is
realizable and hence $\min E^h = 0$ for every $h$. 
\end{abstract}

\maketitle
\tableofcontents

\section{Introduction}

The purpose of this paper is to study effective elastic behaviour
of the incompatibly pre-stressed thin plates $\Omega^h$, characterized by a
Riemann metric $G$ given on their reference configuration. The incompatibility is measured
through the energy $E^h$ given below (sometimes called the ``non-Euclidean''
elastic energy).

We will be concerned with the regime of
curvatures of $G$ which yields the incompatibility rate of order
higher than $h^2$, in plate's thickness $h$. Indeed, in paper \cite{BLS} we
analyzed the scaling $\inf E^h\sim h^2$ and proved that it only occurs
when the metric $G_{2\times 2}$ on the mid-plate  can be
isometrically immersed in $\mathbb{R}^3$ with the regularity $W^{2,2}$
and when, at the same time, the three appropriate Riemann
curvatures of $G$ do not vanish identically (for details, see
below). The relevant residual theory, obtained through 
$\Gamma$-convergence, yielded a  bending Kirchhoff-like residual energy.

In the present paper we assume that:
\begin{equation}\label{scale}
h^{-2}\inf E^h \to 0
\end{equation} 
and prove that the only nontrivial residual theory in this regime is
a von K\'arm\'an-like energy, valid when $\inf E^h\sim h^4$. It further
turns out that this scaling is automatically implied by
(\ref{scale}) and $\inf E^h\neq 0$. Indeed, we show that if
(\ref{scale}) then $h^{-4}\inf E^h\leq C$, and that $h^{-4}\inf
E^h\to 0$ if and only if $G$ is immersible whereas trivially $\min E^h = 0$ for all $h$. 

This scale separation is contrary to the findings of
\cite{FJMhier} valid in the Euclidean case of $G=\mbox{Id}_3$, where
the residual energies are driven by 
presence of applied forces $f^h\sim h^\alpha$. In that context, three
distinct limiting theories have been obtained for $E^h\sim h^\beta$
with $\beta>2$ (equivalently $\alpha>2$). Namely: $\beta\in (2,4)$ corresponded to the
linearized Kirchhoff (nonlinear bending)  model subject to a nonlinear
constraint on the displacements, $\beta=4$ to the classical  von K\'arm\'an 
model, and $\beta>4$ to the linear elasticity. The present results are also contrary to
the higher order hierarchy of scalings and of the resulting elastic
theories of shells, as derived by an asymptotic calculus in
\cite{LPhier}. The difference is due to the fact that while the
magnitude of external forces is adjustable at will, it seems not to be
the case for the interior mechanism of a given metric
$G$ which does not depend on $h$. In fact, it is the curvature tensor of $G$ 
that induces the nontrivial stresses in the thin film. The Riemann
tensor of a three-dimensional metric has only
six independent components, namely the six sectional
curvatures created out of the three principal directions, which further
fall into two categories: including or excluding the thin direction
variable. The simultaneous vanishing of curvatures in each of such categories
correspond to the two scenarios at hand in terms of the scaling of the
residual energy. 

%In what follows, we introduce basic notation and recall the main scaling
%observations in \cite{BLS} and \cite{lepa}. We then make precise statements
%indicated above, to be proven in the following sections.

\subsection{Some background in dimension reduction for thin structures}
Early attempts for replacing the three-dimensional model of a thin
elastic structure with planar mid-surface at rest,  by a
two-dimensional model, were based on {\it a priori} simplifying assumptions
on the deformations and on the stresses.  Later, the natural idea of
using the thickness as a small parameter and of establishing a limit
model was largely explored; we refer in particular to the works by 
Ciarlet and Destuynder who set the method in the appropriate
framework of the weak formulation of the boundary value problems
\cite{CD-lin, CD-nonlin},  proved convergence to the linear plate
model \cite{Destuynder-conv}  in the context of linearized elasticity,
and obtained formally the von 
K\'arm\'an plate model from finite elasticity \cite{C-vK}. See also
\cite{Raoult-annali, Raoult-etat} for  the time-dependent case and
convergence results, and \cite{Cbook-plates} for a comprehensive list
of references.  

The issue of deriving two-dimensional models valid for large
deformations, by means of an asymptotic formalism, was subsequently tackled by
Fox, Simo and the second author in \cite{FRS}. They showed, in the
context of the Saint Venant-Kirchhoff materials subject to appropriate
boundary conditions, how to recover a hierarchy of four
models. This hierarchy, driven by the order of magnitude of the
applied loads, consisted of: the nonlinear membrane model, the
inextensional bending model, the von K\'arm\'an model and the linear
plate model. The models thus obtained %by asymptotic expansions  
%performed  on the system of  balance and constitutive equations 
still required a justification through rigorous convergence results.  In \cite{LR1}, Le
Dret and the second author used  the variational point of view and
proved  $\Gamma$-convergence  of the three-dimensional elastic energies to a nonlinear
membrane energy, valid for loads of magnitude of order $1$. We remark
that the expression of the
limiting stored energy therein consisted of quasiconvexification  of the three-dimensional
energy, first minimized with respect to the normal stretches. This
allowed to recover the degeneracy under compression; a feature that is otherwise
missed by formal expansions. We further mention that for $3$d$\to 1$d reduction a
similar point of view had been introduced by Acerbi, Buttazzo and
Percivale in \cite{ABP}.  

A key-point for deriving rigorously the above mentioned 
nonlinear bending model  has been the geometric
rigidity result due to  Friesecke, James and M\"uller \cite{FJMgeo}. 
% A consequence of the same result is used in the present paper. 
In the similar spirit, the same authors justified the von
K\'arm\'an model, the linear model  \cite{FJMhier} and also they introduced
novel intermediate models, in particular in the range of energies -- or
equivalently of loadings -- between the scaling responsible for
bending ($\beta =2$)
and the von K\'arm\'an scaling ($\beta =4$).  In this range of
models, the three-dimensional stored energy appears in the limit stored
energy  through its second derivative at rest. Scaling the energy with
exponents $\beta$ other than integers had been explored for the
membrane to bending  range  in \cite{ConMag} leading to
convergence results for $0 < \beta < 5/3$  while the regime $5/3 \leq
\beta < 2$ remains open and is 
conjectured to be relevant for crumpling of elastic
sheets. Other significant extensions concern derivation of limit theories
for incompressible materials \cite{3a, 4a, 28a, M4}, for heterogeneous
materials \cite{S}, through establishing convergence of
equilibria rather than strict minimizers \cite{20a, 22a, M1, M3, M2}, 
and finally for shallow shells \cite{lemapa2}. 

Extension of the above variational method valid in the framework of
the large deformation model  was conducted in parallel for slender
structures whose midsurface at rest is non-planar. 
The first result by the second author and Le Dret  \cite{LR2} relates to scaling $\beta = 0$ and models membrane
shells: the limit  stored energy depends then only on the stretching and shearing produced by the
deformation on the midsurface. Another study is due to Friesecke, James, Mora
and M\"uller \cite{FJMM_cr} who analyzed the case $\beta  = 2$. This scaling corresponds to a flexural
shell model, where the only admissible deformations are those
preserving the midsurface  metric. The  limit  energy
depends then on the change of curvature produced 
by the  deformation. Further, the first author, Mora and Pakzad derived the relevant linear
theories ($\beta>4$) and the von K\'arm\'an-like theories ($\beta=4$) in  
\cite{lemopa1}, and subsequently proceeded to finalize the analysis for 
elliptic shells in the full regime $\beta>2$ in \cite{lemopa3}. A similar analysis
has been performed in case of the developable shells in \cite{holepa}
leading to the proof of the collapse of all residual theories to the
linear theory when $\beta>2$. Following these findings,
a conjecture was made in \cite{LPhier} about the infinite
hierarchy of shell models and the various possible limiting scenarios
differentiatied by rigidity properties of shells.   Let us recall that
a comprehensive body of work had been previously devoted to the
asymptotic derivation of shell models  in the small displacement
regime under clear hypotheses on the model taken for granted,
three-dimensional or already two-dimensional and containing the thickness as
a parameter. Several models were recovered by Ciarlet and coauthors
\cite{CL1, CLg, CLM}, by Destuynder \cite{Destuynder-etat,
  Destuynder-class} 
and by Sanchez-Palencia  and coauthors \cite{Sanchez-coq1,
  Sanchez-coq2, CS-anis, CS-stiff, Miara-Sanchez}. Sanchez-Palencia, in 
particular, theorized the role and interplay of the midsurface
geometry and of the boundary conditions \cite{Sanchez3d}, as well as
underlined the singular perturbation behavior. We refer to \cite{Cbook-shells} for many additional references.

%All the above mentioned theories (as well as the subsequent results in this paper)
%should be put in contrast with a large body of literature, devoted to derivations
%starting from $3$-dimensional linear elasticity (see \cite{Cbook} and references
%therein) and by asymptotic expansion method (see e.g. \cite{FRS}). 

\medskip

Most recently, there has been a sustained interest in studying similar problems
where the shape formation is not driven by exterior forces but rather 
by the internal prestrain caused by {\it e.g.}  growth, swelling,
shrinkage or plasticity \cite{sharon, ESK1, 22a}.  Variants of a thin
plate theory can be used to study the self-similar structures which form due to 
variations in an intrinsic metric of a surface that is asymptotically flat at
infinity \cite{1b}, and also in the case of a circular disk with
edge-localized growth \cite{ESK1},  or in the shape of a
long leaf \cite{18b}. Ben Amar and coauthors  formally
derived a variant of the F\"oppl-von K\'arm\'an equilibrium equations
from finite incompressible elasticity \cite{DB, DCB}: they use the
multiplicative decomposition of the gradient proposed in \cite{RHM}
similar to ours. They take cockling of paper, grass blades and 
sympatelous flowers as examples \cite {DCB, BMT}.  

A systematic study of the
possible limit problems when a target metric is prescribed was
undertaken by the first author and collaborators: a  residually strained version of the Kirchhoff 
theory for plates was, for the first time, rigorously derived in
\cite{lepa} under the assumption that the target 
metric is independent of thickness. This analysis was completed in 
\cite{BLS}, resulting in a necessary and sufficient condition that the elastic prestrained energy scales as $h^2$.  
The object of the present paper is to study higher order prestrains.

Let us also mention that in \cite{lemapa2, lemapa2new, LOP} similar derivations were carried
out under a different assumption on the 
asymptotic behavior of the prescribed metric, which also implied
energy scaling $h^\beta$ in different regimes of $\beta>2$. In
\cite{lemapa2} it was shown that the resulting equations  
are identical to those postulated to account for the effects of growth
in elastic plates \cite{18b} and used to describe the shape of a long
leaf. In \cite{LOP} a model with a  Monge-Amp\`ere constraint was derived and analysed
from various aspects. 
%All these results are limited to the case  
%when a decomposition of the deformation gradient into an elastic and inelastic part
%can be carried out (see \cite{RHM} ) - this requires that it is possible to separate out a reference
%configuration, and is thus most relevant for the description of plant morphogenesis.
Other results  concerning the energy scaling for the materials with
residual strain are derived in \cite{KoBe}, where by imposing
suitable boundary data, conditions of \cite{lepa, BLS} are not satisfied and hence the residual energy
scales larger than $h^2$, depending on the type of these boundary  data (see also \cite{22a}).  

\subsection{The set-up and notation}

Let $\Omega$ be an open, bounded, smooth and simply connected subset
of $\mathbb{R}^2$. For $0<h\ll 1$ we consider thin films $\Omega^h$
around the mid-plate $\Omega$:
\begin{equation}
\Omega^h =  \big\{x=(x',
  x_3);~~ x'\in\Omega, ~~x_3\in(-h/2, h/2)\big\}.
\end{equation}
Let $G:\bar{\Omega}^h\rightarrow \mathbb{R}^{3\times 3}$ be a given smooth
Riemann metric on $\Omega^h$, uniform through the thickness:
$$G(x',x_3)=G(x') \quad \mbox{for every } (x',x_3)\in \Omega^h,$$ 
and let $A=\sqrt{G}$ denote the unique positive definite symmetric square root of $G$.
%In what follows, by $G_{2\times 2}(x')$ we denote the two-dimensional
%principal minor of $G(x')$, which defines a smooth Riemann metric on $\Omega$.
Consider the energy functional $E^h: W^{1,2}(\Omega^h,
\mathbb{R}^3)\rightarrow \bar{\mathbb{R}}_+$ defined as: 
\begin{equation}\label{functional}
E^h(u^h)=\frac{1}{h}\int_{\Omega^h} W(\nabla u^hA^{-1})\dx.
\end{equation}
The nonlinear elastic energy density $W:\mathbb{R}^{3\times
  3}\rightarrow \bar{\mathbb{R}}_{+}$ is a Borel measurable function, 
assumed to be $\mathcal{C}^2$ in a neighborhood
of $SO(3)$ and to satisfy, for every $F\in \mathbb{R}^{3 \times 3}$,
every $R\in SO(3)$ and with a uniform constant $c>0$, the conditions: 
\begin{equation}\label{properties}
\begin{split}
& W(R)=0, \qquad W(RF)=W(F), \qquad W(F)\geq c\dist^2\pare{F, SO(3)}.
\end{split}
\end{equation}
The first condition states that the energy of a rigid motion is
$0$, while the second is the frame invariance. They imply that
$DW(\mbox{Id}_3) =0$ and that $D^2W(\mbox{Id}_3)(S,\cdot)=0$ for all
skew symmetric matrices $S\in so(3)$. The third assumption above reflects the
quadratic growth of the density $W$ away from the energy well
$SO(3)$. Note that these assumptions are not contradictory with
the physical condition $W(F)=\infty$ for $\det F\leq 0$.

Throughout the paper, we use the following notation. Given a matrix
$F\in\mathbb{R}^{3\times 3}$, we denote 
its transpose by $F^t$, its symmetric part by $\mbox{sym} F =
\frac{1}{2}(F + F^t)$, and its skew part by $\mbox{skew} F = F - \mbox{sym} F$. 
By $SO(n) = \{R\in\mathbb{R}^{n\times n}; ~ R^t = R^{-1} \mbox{ and }
  \det R=1\}$ we denote the group of special rotations, while $so(n) =
\{F\in \mathbb{R}^{n\times n}; ~ \mbox{sym}F=0\}$ is the space of
skew-symmetric matrices.
We use the matrix norm $|F| = (\mbox{trace}(F^t F))^{1/2}$, which
is induced by the inner product $\langle F_1 : F_2\rangle =
\mbox{trace}(F_1^t F_2)$. The $2 \times 2$ principal minor of a matrix
$F\in\mathbb{R}^{3\times 3}$ is denoted by $F_{2\times 2}$.
Conversely, for a given $F_{2\times 2} \in \mathbb{R}^{2\times 2}$, the $3 \times 3$ matrix with
principal minor equal $F_{2\times 2}$ and all other entries equal to
$0$, is denoted by $(F_{2\times 2})^*$. All limits are
taken as the thickness parameter $h$ vanishes, i.e. when $h\to 0$. Finally, by $C$ we denote any
universal constant, independent of $h$.

\subsection{Some previous directly related results}
It has been proved in \cite{lepa} that: 
$$\inf_{u^h\in W^{1,2}(\Omega^h, \mathbb{R}^3)} E^h(u^h) = 0$$ 
if and only if the Riemann curvature tensor of $G$ vanishes identically in $\Omega^h$, i.e.:
$\mbox{Riem}(G)\equiv 0$, and when (equivalently) the infimum above is achieved
through a smooth isometric immersion $u^h$ of the metric $G$ on
$\Omega^h$. Further, in \cite{BLS} it is proved that: 
$$ \lim_{h\to 0} \frac{1}{h^2} \inf E^h=0$$
if and only if the following Riemann curvatures of $G$ vanish identically:
\begin{equation}\label{vanish}
R_{1212}=R_{1213}=R_{1223} \equiv 0 \quad \mbox{in } \Omega^h.
\end{equation}
More generally, the limit behavior of the rescaled energies $h^{-2}E^h$ has been investigated in \cite{BLS}
and it has been proved that their $\Gamma$-limit is given by the functional:
$$\mathcal{I}_2(y) = \frac{1}{24}\int_\Omega \mathcal{Q}_{2,A}\big(x',
(\nabla y)^t\nabla \vec b\big)~\mbox{d}x',$$
effectively defined on the set of all $y\in
W^{2,2}(\Omega,\mathbb{R}^3)$ such that $(\nabla y)^t\nabla y =
G_{2\times 2}$. The quadratic forms $\mathcal{Q}_{2, A}(x', \cdot)$
are given by means of the energy density $W$ as in (\ref{Q2A}). The Cosserat vector $\vec b\in
W^{1,2}\cap L^\infty(\Omega,\mathbb{R}^3)$ is uniquely determined from
the isometric immersion $y$ by:
\begin{equation}\label{Q_0a}
Q^t Q=G \quad \mbox{where} \quad Qe_1=\partial_1 y, \;\;\; Qe_2=\partial_2
y, \;\;\;  Qe_3=\vec b, \quad \mbox{with } \det{Q}>0.
\end{equation}
Observe that the functional $\mathcal{I}_2$ is a Kirchhoff-like fully nonlinear
bending energy, which in case of $Ge_3 = e_3$ reduces to the classical
bending content quantifying the second fundamental form $(\nabla y)^t\nabla
\vec b= (\nabla y)^t\nabla \vec N$ on the deformed surface $y(\Omega)$ with the
unit normal vector $\vec N$.

We recall that by Theorems 5.3, 5.5 and Corollary 5.4 in \cite{BLS},
the negation of condition \eqref{vanish} is equivalent to $\min
\mathcal{I}_2>0$. For this reason, \eqref{vanish} 
is equivalent to the existence of a, necessarily unique and smooth, isometric immersion
$y_0:\Omega \rightarrow \mathbb{R}^3$ of $G_{2\times 2}$, such that:
\begin{equation}\label{sym0}
\left\{\begin{split}
& (\nabla y_0)^t\nabla y_0 = G_{2\times 2}\\
& \mbox{sym}\big((\nabla y_0)^t \nabla\b\big)=0. 
\end{split}\right. %\qquad \mbox{  in}\;\; \Omega,
\end{equation}
Above, the smooth vector field $\vec b_0$ and the smooth matrix field
$Q_0$ are given as in (\ref{Q_0a}):
\begin{equation}\label{Q_0}
Q_0^t Q_0=G, \;\;\; Q_0e_1=\partial_1 y_0, \;\;\; Q_0e_2=\partial_2
y_0 \;\;\;\mbox{and} \;\;\; Q_0e_3=\b \quad \mbox{with } \det{Q_0}>0.
\end{equation}
Equivalently, denoting the inverse matrix $G^{-1}=[G^{ij}]_{i,j:1.. 3}$, we have:
\begin{equation}\label{exprb}
\b=-\frac{1}{G^{33}}\big(G^{13}\partial_1y_0+G^{23}\partial_2y_0\big)+\frac{1}{\sqrt{G^{33}}}\vec N.
\end{equation}
Uniqueness of the immersion $y_0$ in (\ref{sym0}) follows from Theorem 5.3
in \cite{BLS} which shows that the second fundamental form of the
surface $y_0(\Omega)$ is given in terms of $G$. Therefore, both fundamental forms are known.
Also, the second equation in (\ref{sym0}) comes from the fact that the kernel
of each quadratic form $\mathcal{Q}_{2,A}$ consists of $so(2)$.
% This is a consequence of the fact that $D^2W(\mbox{Id}_3)$ vanishes on
% $so(3)$ in view of the frame invariance of $W$, and the fact that it is positive
% definite on the symmetric matrices in view of the quadratic lower
% bound in (\ref{properties}).

\subsection{New results of this work}
In view of the above statements, in this paper we investigate
smaller energy scalings and the limiting behaviour of the minimizing configurations to
$E^h$ under condition \eqref{vanish}. We first prove (in Lemma \ref{scalingh4}) that (\ref{vanish}), which
as we recall is equivalent to (\ref{scale}), implies:
$$\inf E^h\leq Ch^4.$$
We then derive (in Theorem \ref{lowerless4} and Theorem
\ref{upperless4})  the $\Gamma$-limit $\mathcal{I}_4$ of the rescaled
energies $h^{-4}E^h$, together with their compactness properties. 

Namely, let $y_0$ be the unique
immersion satisfying (\ref{sym0}), where $\vec b_0$ is as in (\ref{Q_0}).
Let $\d:\Omega\rightarrow\mathbb{R}^3$ be the smooth vector field
given in terms of $y_0$ by:
\begin{equation}\label{d}
\langle Q_0^t\d, e_1\rangle = -\langle\partial_1\b, \b\rangle, \;\;\;
\langle Q_0^t\d, e_2\rangle = -\langle\partial_2 \b, \b\rangle,\;\;\;
\langle Q_0^t\d, e_3\rangle = 0.
\end{equation}
The limit $\mathcal{I}_4$ is then the following energy functional: %$\mathcal{I}_4(V, \mathbb{S})$ in:
\begin{equation}\label{limit_fun2}
\begin{split}
\mathcal{I}_{4}(V,\mathbb{S}) = ~ & \frac{1}{2}\int_{\Omega} \mathcal{Q}_{2,A}
\pare{x', \mathbb{S}+\frac{1}{2} (\nabla V)^t\nabla V + \frac{1}{24} (\nabla\b)^t\nabla\b }\dx ' \\
& + \frac{1}{24}\int_{\Omega} \mathcal{Q}_{2,A}\pare{x',  (\nabla
  y_0)^t \nabla \vec p   +   (\nabla V)^t\nabla\b } \dx ' \\
& +\frac{1}{1440}\int_{\Omega} \mathcal{Q}_{2,A}\pare{ x', (\nabla
    y_0)^t\nabla\d + (\nabla\b)^t\nabla\b  } \dx',
\end{split}
\end{equation}
acting on the space of finite strains:
$$\mathbb{S}\in \mathrm{cl}_{L^2}\big\{\mbox{sym}\big((\nabla y_0)^t\nabla w);~
w\in W^{1,2}(\Omega,\mathbb{R}^3\big)\big\}$$ 
and the space of first order infinitesimal isometries:
$$V\in W^{2,2}(\Omega, \mathbb{R}^3) \quad \mbox{ such that: } \quad
\mbox{sym}\big((\nabla y_0)^t\nabla V)_{2\times 2} = 0,$$ 
where the vector field $\vec p\in W^{1,2}(\Omega,\mathbb{R}^3)$ is uniquely
associated
with each $V$ by:
$ (\nabla y_0)^t \vec p =-(\nabla V)^t \b$ and  $\inn{\b}{\vec p} =0.$

The spaces consisting of  $\mathbb{S}$ and $V$ contain the information
about the admissible error displacements,
relative to the leading order immersion $y_0$, under the energy scaling
$E^h\sim h^4$. We discuss the
geometrical significance of $V$ and $\mathbb{S}$ and of various
bending and stretching tensors in the first two terms of 
$\mathcal{I}_4(V, \mathbb{S})$ in section \ref{secexp}.
We further prove in Theorem \ref{strangeresult} that the last term in
(\ref{limit_fun2}), which is obviously constant and as such does not
play a role in the minimization process, is precisely given by the only potentially nonzero
(in view of (\ref{vanish})) curvatures of $G$, namely:
\begin{equation*}
\sym \big((\nabla y_0)^t\nabla \vec d_0\big)  + (\nabla\vec b_0)^t \nabla\vec
b_0 = \left[\begin{array}{cc} R_{1313} & R_{1323} \\ R_{1323} & R_{2323}\end{array}\right].
\end{equation*}
We may thus write, informally:
\begin{equation*}
\begin{split}
\mathcal{I}_{4}(V,\mathbb{S}) = ~ & \frac{1}{2}\int_{\Omega} \mathcal{Q}_{2,A}
(x', \mbox{stretching of order } h^2) \dx ' 
+ \frac{1}{24}\int_{\Omega} \mathcal{Q}_{2,A}(x', \mbox{bending of order } h) \dx ' \\
& \qquad +\frac{1}{1440}\int_{\Omega} \mathcal{Q}_{2,A} (x',
  \mbox{Riemann curvature of } G ) \dx'.
\end{split}
\end{equation*}
In particular, since all three terms above are nonnegative,
this directly implies that the condition
$\lim_{h\to 0}\frac{1}{h^4}\inf E^h = 0,$
which is equivalent to $\min \mathcal{I}_4=0$, is further
equivalent to the immersability $G$, i.e. the vanishing of
all its Riemann curvatures $\mbox{Riem}(G) \equiv 0$ in $\Omega^h$.

\subsection{Acknowledgments}
M.L. was partially supported by the NSF grant DMS-0846996
and the NSF grant DMS-1406730. A part of this work has been carried out while
the first author visited the second author at the Universit\'e Paris Descartes,
whose support and warm hospitality are gratefully acknowledged.

\section{The scaling and approximation lemmas}

We first introduce the following notation. Let $B_0(x')$ be the matrix field satisfying:
\begin{equation}\label{B_0}
B_0e_1=\partial_1 \b, \;\;\; B_0e_2=\partial_2 \b \;\;\;\mbox{and}
\;\;\; B_0e_3=\d,
\end{equation}
where $\vec d_0$ is given by (\ref{d}). Observe that in this way  $Q_0^tB_0$ is skew symmetric. Indeed, it has
the following block form:
\begin{equation}\label{skew}
Q_0^tB_0=
\left[\begin{array}{c|c}
\pare{\nabla y_0}^t \nabla \b &  \pare{\nabla y_0}^t\d \\ \hline
(\b)^t\nabla \b & \langle \b,\d\rangle
\end{array}\right]
\end{equation}
and by \eqref{sym0} we see that $\pare{\nabla y_0}^t \nabla
\b\in so(2)$ is skew symmetric, while by \eqref{d}: $\pare{\nabla
  y_0}^t\d=-(\nabla \b)^t\vec b_0$ and $\langle\b,\d\rangle = 0$. 

\begin{lemma}\label{scalingh4}
Condition (\ref{vanish}) implies: $\displaystyle{\inf_{W^{1,2}(\Omega^h,\mathbb{R}^3)} E^h \leq Ch^4}$.
\end{lemma}
\begin{proof}
Let us construct a sequence $u^h\in W^{1,2}(\Omega^h, \mathbb{R}^3)$ that has low energy. Let:
\begin{equation}\label{change2}
 u^h(x',x_3)= y_0(x')+x_3\b(x')+\frac{x_3^2}{2}\d(x'),
\end{equation}
in fact each $u^h$ is the restriction on its domain $\Omega^h$ of the same deformation.
We have:
\begin{equation*}
\nabla u^h(x',x_3)= Q_0(x')+x_3B_0(x')+\frac{x_3^2}{2}D_0(x') ,
\end{equation*}
where the matrix field  $D_0(x')\in\mathbb{R}^{3\times 3}$ is given through:
\begin{equation*}
D_0(x')e_1=\partial_1 \d, \;\;\;
D_0(x')e_2=\partial_2 \d,\;\;\;
D_0(x')e_3=0,
\end{equation*}
so that:
\begin{equation*}
\nabla u^h A^{-1} = Q_0 A^{-1} +x_3B_0 A^{-1}  +\frac{x_3^2}{2}D_0 A^{-1}.
\end{equation*}
For brevity, denote $F^h=\nabla u^h A^{-1}$. Obviously, $F^h$ decomposes as:
\begin{equation}\label{burp}
F^h(x',x_3)=Q_0A^{-1}(x')(\Id+x_3S(x')+x_3^2T(x'))=(Q_0A^{-1}(x')) G^h(x',x_3)
\end{equation}
with $S= A^{-1}Q_0^t B_0 A^{-1}$, $T= \frac{1}{2}A^{-1}Q_0^t D_0
A^{-1}$ and  $G^h= \Id+x_3S +x_3^2T$. Since $Q_0A^{-1}\in SO(3)$ by construction, frame indifference implies that
$W(F^h)= W \pare{((G^h)^tG^h)^{1/2}}$.
Note that $S$ is skew symmetric, because $Q_0^tB_0$ is skew
symmetric. Therefore, $(G^h)^tG^h$ and the expansion of its square
root do not contain terms linear in $x_3$. Indeed, letting $K=
T+T^t-S^2$: 
\begin{equation*}
((G^h)^t G^h)(x',x_3)=\Id+x_3^2K(x')+  \bigo{x_3^3}
\end{equation*}
and:
\begin{equation*}
((G^h)^tG^h)^{1/2}(x',x_3)=\Id+\frac{x_3^2}{2}K(x')+ \bigo{x_3^3}.
\end{equation*}
As a consequence, using $W(\Id)=0$ and $DW(\Id)=0$, we obtain:
\begin{equation*}
W(F^h)= W\pare{((G^h)^t G^h)^{1/2}}= \frac{x_3^4}{8} D^2W(\Id)(K,K) + \bigo{x_3^5}.
\end{equation*} 
Using (\ref{functional}), we get
\begin{equation*}
E^h(u^h)=\frac{1}{h}\int_{\Omega^h} W(F^h)~\mbox{d}x\leq Ch^4,
\end{equation*}
which accomplishes the proof of the lemma.
\end{proof}

\bigskip

In Lemma \ref{scalingh4} above, we constructed deformations whose
gradient was sufficiently close to $Q_0+x_3B_0$, to provide the energy
of the order $h^4$. Conversely, in Corollary \ref{lemma_estimate2}
below, we establish that the gradients of deformations $u^h$ whose
energy scales like $h^4$, are close to $Q_0+x_3B_0$ modulo local
multiplications by $R^h(x')\in SO(3)$.  Corollary
\ref{lemma_estimate2} makes this statement precise and gives an
estimation on $\nabla R^h$ as well. 

\bigskip

For any $\mathcal{V}$ which is an open subset of $\Omega$, we let
$\Vh=\mathcal{V} \times (-h/2,h/2)$ and  we define the local energy
functional by:  
\begin{equation*}
 E^h(u^h,\Vh)=\frac{1}{h}\int_{\Vh} W(\nabla u^h A^{-1})\dx.
\end{equation*}

\begin{lemma}\label{Lemma_approx}
Assume (\ref{vanish}). There exists a constant $C>0$ with the following property. For any $u^h\in
W^{1,2}(\Vh,  \mathbb{R}^3)$, there exists $\bar R^h\in SO(3)$ such
that:  
\begin{equation}
\frac{1}{h}\int_{\Vh} \abs{\nabla u^h(x) -\bar R^h ( Q_0(x')+x_3B_0(x'))}^2 \dx
\leq C \left(  E^h(u^h, \Vh) +h^3 |\Vh| \right).
\end{equation}
The constant $C$ is  uniform
for all  $\mathcal{V}^h$ which are bi-Lipschitz equivalent with controlled Lipschitz constants.
\end{lemma}

\begin{proof}
By assumption (\ref{properties}), we have:
\begin{equation}\label{minE}
E^h(u^h, \Vh) \geq \frac{c}{h}\int_{\Vh} \dist^2\pare{ \nabla u^h A^{-1}, SO(3) } \dx.
\end{equation}
This suggests performing a change of variables in order to use the
nonlinear geometric rigidity estimate \cite{FJMhier}. 
%In \cite{lepa}, we used the simple change of variables $u \circ
%A^{-1}(x_0)$. Here, we need more precise a result and, 
For any $u^h\in W^{1,2}(\Vh,  \mathbb{R}^3)$,  we let $v^h=u^h\circ
Y^{-1}$ with $Y:\Vh \rightarrow Y(\Vh)=\Uh\subset\mathbb{R}^3$ given
as in (\ref{change2}), namely: 
\begin{equation*}
 Y(x',x_3)= y_0(x')+x_3\b(x')+\frac{x_3^2}{2}\d(x').
\end{equation*}
Obviously, $v^h\in W^{1,2}(\Uh, \mathbb{R}^3)$ and:
\begin{equation}\label{changenab}
\nabla u^h A^{-1}(x', x_3) = \nabla v^h(z) (\nabla Y A^{-1}) (x',x_3), \qquad z:= Y(x', x_3).
\end{equation}
Let $S'= B_0 Q_0^{-1}$ and $T'= \frac{1}{2}D_0 Q_0^{-1}$. Note that
$S' =B_0 Q_0^{-1} = Q_0^{-1, t} (Q_0^tB_0 Q_0^{-1}) = - Q_0^{-1,t}
B_0^t $ in view of $Q_0^tB_0\in so(3)$. Therefore
$S'\in so(3)$. Computations as in Lemma \ref{scalingh4} now give: 
\begin{equation}\label{grady}
\nabla Y(x',x_3)= Q_0(x')+x_3B_0(x')+\frac{x_3^2}{2}D_0(x'),
\end{equation}
and:
\begin{equation*}
\nabla Y A^{-1}  =\pare{\Id+x_3S'(x')+x_3^2T'(x')} (Q_0A^{-1}).
\end{equation*}
We see that for $h$ small, $\det (\nabla Y A^{-1})>0$. Further, the
left polar decomposition $\nabla Y A^{-1}= \pare{\nabla Y
  A^{-1}(\nabla Y A^{-1})^t}^{1/2} R$, allows us to write:
\begin{equation*}
\nabla Y A^{-1}=  (\Id+x_3^2 M(x', x_3)) R(x', x_3),
\end{equation*}
where $M=\bigo{1}$ is a symmetric matrix field and $R\in SO(3)$. Again, the symmetric
term does not contain any term linear in $x_3$. Therefore:
\begin{equation*}\label{est_O}
\begin{split}
 \dist &\pare{\nabla v^h \nabla Y A^{-1}, SO(3)}  = \dist  \pare{\nabla v^h  (\Id+x_3^2 M) R, SO(3)}\\
 &= \dist  \pare{\nabla v^h  (\Id+x_3^2 M), SO(3)} \geq c \dist \pare{\nabla v^h, SO(3)(\Id+x_3^2 M)^{-1}}\\
&\geq c \dist \pare{\nabla v^h, SO(3)}+\bigo{x_3^2}.
\end{split}
\end{equation*}
Now, let $J=\abs{\det \nabla Y \circ Y^{-1}}^{-1}$. By (\ref{changenab}) and the above computation:
\begin{equation*}\label{est_part1}
\int_{\Vh} \dist^2\pare{ \nabla u^h A^{-1}, SO(3) }\dx   \geq
c\int_{\Uh}\dist^2 \pare{ \nabla v^h, SO(3) } J\dz - c \int_{\Vh} x_3^4
\dx. 
\end{equation*}
In other words, since $J\geq c>0$:
\begin{equation*}
\frac{1}{h}\int_{\Vh} \dist^2\pare{ \nabla u^h A^{-1}, SO(3) }\dx  +
h^3 |\Vh| \geq \frac{c}{h}\int_{\Uh}\dist^2 \pare{ \nabla v^h, SO(3)} \dz. 
 \end{equation*}
By \cite{FJMhier}, there exists $C>0$ with the following property. For
any $v^h\in W^{1,2}(\Uh, \mathbb{R}^3)$, there exists $\bar R^h\in
SO(3)$ such that: 
\begin{equation*} 
C \int_{\Uh} \dist^2 \pare{\nabla v^h, SO(3)} \dz \geq \int_{\Uh} \abs{\nabla v^h- \bar R^h}^2\dz.
\end{equation*}
The constant $C$ can be chosen uniformly for domains $\Uh$ which are
bi-Lipschitz equivalent with controlled Lipschitz constants. By
(\ref{minE}) and the reverse change of variables which satisfies
$J^{-1}\geq c>0$ and $|\nabla Y|\leq C$, we obtain:
\begin{equation*}\label{est_part2}
C\left(E^h(u^h, \Vh) + h^3  |\Vh|\right) \geq  \frac{1}{h} \int_{\Vh} \abs{\nabla u^h  -Ê\bar R^h \nabla Y}^2 \dx,
\end{equation*}
again with a constant $C$ uniform for domains $\Vh$ that are
bi-Lipschitz equivalent with controlled Lipschitz
constants. This accomplishes the proof in view of  (\ref{grady}).
\end{proof}

\begin{corollary}\label{lemma_estimate2}
Assume (\ref{vanish}) and let $u^h$ be a sequence of deformations such that: 
$$\lim_{h\to 0} {h^{-2}}E^h(u^h)=0.$$  Then, there
exist matrix fields $R^h\in W^{1,2}(\Omega, SO(3))$ such that: 
\begin{equation}\label{estimate1}
\frac{1}{h}\int_{\Omega^h} \abs{\nabu(x)-R^h(x')\pare{Q_0(x')+x_3B_0(x')}}^2 \dx
\leq  C  \pare{E^h(u^h)+h^4}
\end{equation}
and:
\begin{equation}\label{estimate2}
\int_{\Omega}\abs{\nabla R^h(x')}^2\dx' 
\leq \frac{C}{h^2}\pare{E^h(u^h)+h^4}.
\end{equation}
\end{corollary}

The proof follows the lines of \cite{FJMhier, lepa, lemapa1}, with
necessary modifications in view of the expected error of the order
$h^4$. For completeness, we will present the details in the Appendix.

\section{The lower bound}

\begin{theorem}\label{lowerless4}
Let $u^h\in W^{1,2}(\Omega^h, \mathbb{R}^3)$ be a sequence of
deformations satisfying $E^h(u^h)\leq Ch^4$. Then there exists a
sequence of translations $c^h\in\mathbb{R}^3$ and rotations $\bar{R}^h\in
SO(3)$ such that the associated renormalizations:
\begin{equation}\label{yh}
%%%% added label 
y^h(x',x_3)=(\bar{R}^h)^tu^h(x', hx_3)-c^h \in W^{1,2}(\Omega^1, \mathbb{R}^3)
\end{equation}
have the following properties, where $y_0$ and $\vec b_0$ are the
unique solution to (\ref{sym0}) (\ref{Q_0}). All convergences hold up to a subsequence:
\begin{itemize}
\item[(i)]
$ y^h \rightarrow y_0$ in $W^{1,2}(\Omega^1, \mathbb{R}^3)$ and
$\frac{1}{h}\partial_3y^h\rightarrow \vec b_0$ in $L^2(\Omega^1, \mathbb{R}^3)$;
\item[(ii)] the scaled average displacements:
\begin{equation}\label{Vh_defa}
V^h(x')=\frac{1}{h}\fint_{-\frac{1}{2}}^{-\frac{1}{2}} \Big(y^h(x',x_3)-\big (y_0(x')+ h x_3\b(x')\big )\Big)\dxt
\end{equation}
converge in $W^{1,2}(\Omega, \mathbb{R}^3)$ to a limiting field $V\in
W^{2,2}(\Omega, \mathbb{R}^3)$, satisfying the constraint: 
\begin{equation}\label{sym0a}
\sym \pare{(\nabla y_0)^t\nabla V} =0;
\end{equation}
\item[(iii)] the scaled tangential strains:
$$\frac{1}{h} \sym \big( (\nabla y_0)^t\nabla V^h\big)$$
converge weakly in $L^2(\Omega,\mathbb{R}^{2\times 2})$ to some $\mathbb{S}\in
L^2(\Omega,\mathbb{R}^{2\times 2}_{\mbox{sym}})$. 
\item[(iv)] Further, defining the quadratic forms $\mathcal{Q}_{3}(F)=D^2W(\mathrm{Id}_3)(F,F)$ and:
\begin{equation}\label{Q2A}
\mathcal{Q}_{2,A}(x',F_{2\times 2})=\min  \left\{
  \mathcal{Q}_3\big(A(x')^{-1}\tilde{F}A(x')^{-1}\big); ~~
  \tilde{F}\in\mathbb{R}^{3\times 3}\;\; \mbox{with}
  \;\;\tilde{F}_{2\times 2}=F_{2\times 2} \right\}, 
\end{equation}
we have:
\begin{equation}\label{limit_fun}
\begin{split}
\liminf_{h\to 0} \frac{1}{h^4}E^h(u^h) \geq 
\mathcal{I}_{4}(V,\mathbb{S}) = ~ & \frac{1}{2}\int_{\Omega} \mathcal{Q}_{2,A}
\pare{x', \mathbb{S}+\frac{1}{2} (\nabla V)^t\nabla V + \frac{1}{24} (\nabla\b)^t\nabla\b }\dx ' \\
& + \frac{1}{24}\int_{\Omega} \mathcal{Q}_{2,A}\pare{x',  (\nabla
  y_0)^t \nabla \vec p   +   (\nabla V)^t\nabla\b } \dx ' \\
& +\frac{1}{1440}\int_{\Omega} \mathcal{Q}_{2,A}\pare{ x', (\nabla
    y_0)^t\nabla\d + (\nabla\b)^t\nabla\b  } \dx',
\end{split}
\end{equation}
where the vector field $\vec p\in W^{1,2}(\Omega,\mathbb{R}^3)$ is
uniquely associated  with $V$ by:
\begin{equation}\label{def_p}
\begin{cases}
(\nabla y_0)^t \vec p =-(\nabla V)^t \b \\
\inn{\b}{\vec p} =0.
\end{cases}
\end{equation}
\end{itemize}
\end{theorem}

\begin{proof}

{\bf 1.}  Corollary \ref{lemma_estimate2} yields existence of $R^h\in W^{1,2}(\Omega, SO(3))$ such that
(\ref{estimate1}) and (\ref{estimate2}) hold with $Ch^4$ and $Ch^2$ in
their right hand sides, respectively. We rewrite these inequalities for the reader's convenience:
\begin{equation}\label{estimate3}
 \frac{1}{h}\int_{\Omega^h} \abs{\nabu(x)-R^h(x')\pare{Q_0(x')+x_3B_0(x')}}^2 \dx \leq  C  h^4  
\end{equation}
and:
\begin{equation}\label{estimate4}
 \int_{\Omega}\abs{\nabla R^h(x')}^2\dx' 
\leq C h^2.
\end{equation}
To prove the claimed convergence properties for (\ref{yh}), it is
natural  in view of (\ref{estimate3}) to set:
\begin{equation*}
\bar{R}^h = \proj{ \fint_{\Omega^h}\nabla u^h(x) Q_0(x')^{-1}\dx}.
\end{equation*}
This projection is well defined, because for every $x'\in\Omega$, in
view of (\ref{estimate3}):
\begin{equation*}%\label{09a}
\begin{split}
\dist^2&\pare{\fint_{\Omega^h}\nabla u^h Q_0^{-1}\dx, SO(3)  }
\leq \abs{\fint_{\Omega^h}\nabla u^h Q_0^{-1}\dx - R^h(x')}^2  \\
& \qquad \qquad \leq C  \abs{ \fint_{\Omega^h} (\nabla u^hQ_0^{-1}-R^h) \dx}^2 
+ C \abs{\fint_{\Omega^h}R^h\dx - R^h(x')}^2 \\
& \qquad\qquad \leq  C  \abs{\fint_{\Omega^h} \left(\nabla u^h-R^h(Q_0+x_3B_0 )\right) Q_0^{-1}}^2 \dx 
+ C \abs{ R^h(x')-\fint_{\Omega}R^h}^2\\
& \qquad\qquad \leq  C \fint_{\Omega^h} |\nabla u^h-R^h(Q_0+x_3B_0 )|^2 \dx 
+  C |R^h(x')-\fint_{\Omega}R^h|^2\\
& \qquad\qquad \leq  C h^4 + C |R^h(x')-\fint_{\Omega}R^h|^2
\end{split}
\end{equation*}
Now, taking the average on $\Omega$, by the Poincar\'e-Wirtinger
inequality and (\ref{estimate4}), we get: 
\begin{equation*}
\dist^2\pare{\fint_{\Omega^h}\nabla u^h Q_0^{-1}\dx, SO(3)  } \leq
Ch^4 +C\int_\Omega |\nabla R^h|^2 \leq Ch^2,
\end{equation*}
which proves  that the average $\fint_{\Omega^h}\nabla u^h Q_0^{-1}\dx$ is close to $SO(3)$  and that:
\begin{equation}\label{pso}
|\fint_{\Omega^h}\nabla u^hQ_0^{-1}\dx -  \bar{R}^h|^2\leq Ch^2.
\end{equation}
Moreover:
\begin{equation}\label{RhRha}
\begin{split}
& \fint_{\Omega}|R^h-\bar{R}^h|^2\dx  =  \fint_{\Omega^h}|R^h-\bar{R}^h|^2\dx \\
& \leq C\fint_{\Omega^h} \left(|R^h-\fint_{\Omega}R^h |^2
 + |( \fint_{\Omega}R^h) - \fint_{\Omega^h}\nabla u^h Q_0^{-1} |^2 \right)
+ \fint_{\Omega^h} |\bar{R}^h  - \fint_{\Omega^h}\nabla u^hQ_0^{-1} |^2 \\ & 
\leq C\fint_{\Omega^h}|\nabla R^h|^2\dx + C\fint_{\Omega^h}|\nabla u^h -
R^h(Q_0 + x_3 B_0)|^2~\mbox{d}x + Ch^2 \leq Ch^2,
\end{split}
\end{equation}
 where the last estimate follows by (\ref{estimate3}), (\ref{estimate4}) and (\ref{pso}). 
 
Let now $c^h\in\mathbb{R}^3$ be such that $\int_{\Omega} V^h=0$ where
$V^h$ is defined as in (\ref{Vh_defa}). 
Denote by $\nabla_h y^h$ the matrix whose columns are given by
$\partial_1 y^h$, $\partial_2 y^h$ and $\partial_3 y^h/h$. Obviously:
\begin{equation}\label{yu}
\nabla_h y^h (x',x_3) =  (\bar{R}^h)^t\nabla u^h (x', hx_3).
\end{equation}
Observe that:
\begin{equation*}
\begin{split}
& \int_{\Omega^1}| \nabla_h y^h-Q_0 |^2\dx
 \leq C \fint_{\Omega^h} | \nabla u^h-\bar{R}^hQ_0 |^2\dx \\
& \qquad \leq C (\fint_{\Omega^h} |\nabla u^h-R^h(Q_0+x_3B_0)|^2 \dx
+\fint_{\Omega^h}|x_3R^hB_0|^2\dx + \fint_{\Omega^h}
|R^h-\bar{R}^h|^2\dx ) \leq C h^{2}
\end{split}
\end{equation*}
by (\ref{estimate3}) and  \eqref{RhRha}.
Therefore, $\nabla_h y^h$ converges in $L^2(\Omega^1)$ to
$Q_0$. Further, the sequence $\{y^h\}$ is bounded in $W^{1,2}(\Omega^1)$, 
by the choice of $c^h$. Passing to a subsequence we get that
$y^h$ converges weakly in $W^{1,2}(\Omega^1)$ and in view of the
strong convergence of $\nabla y^h$ we have: 
\begin{equation*}
\begin{split}
y^h \rightarrow y_0 \quad \mbox{in}\quad W^{1,2}(\Omega^1,
\mathbb{R}^3) \quad \mbox{and} \quad
\frac{1}{h}\partial_3 y^h \rightarrow \b \quad \mbox{in}\quad L^2(\Omega^1, \mathbb{R}^3).
\end{split}
\end{equation*}

\medskip

{\bf 2.} Note that, for every $x'\in\Omega$:
\begin{equation}\label{nablaVa}
\begin{split}
& \nabla V^h(x') = \frac{1}{h} \pare{ \fint_{-1/2}^{1/2} \nabla_h y^h(x)-Q_0(x')\dxt  }_{3\times 2} \\
& = \frac{1}{h} \pare{  \fint_{-1/2}^{1/2} \nabla_h y^h-(\bar{R}^h)^tR^h(Q_0+hx_3B_0)\dxt  }_{3\times 2}
+ \frac{1}{h} \pare{  ((\bar{R}^h)^tR^h-\Id)Q_0  }_{3\times 2} \\
& = I_1^h+I_2^h.
\end{split}
\end{equation}
The first term above converges to $0$. Indeed:
\begin{equation}\label{est_I1a}
\begin{split}
\norm{I_1^h}_{L^2(\Omega)}^2 &\leq \frac{C}{h^2} \fint_{\Omega^1}|
(\bar{R}^h)^t\nabla u^h(x', hx_3)-
(\bar{R}^h)^tR^h(Q_0(x')+hx_3B_0)|^2 \dx \\ 
&\leq \frac{C}{h^2} \fint_{\Omega^h} | \nabla u^h(x', x_3)-R^h(Q_0+x_3B_0) |^2\dx
\leq Ch^2.
\end{split}
\end{equation}
Towards estimating the second term in (\ref{nablaVa}), denote: 
\begin{equation*}
S^h=\frac{1}{h}((\bar{R}^h)^tR^h-\Id).
\end{equation*}
By \eqref{RhRha} and \eqref{estimate4}, it follows that:
\begin{equation*}
\lVert S^h\rVert_{L^2(\Omega)}^2 \leq \frac{C}{h^2} \int_{\Omega}|R^h-\bar{R}^h|^2
\leq C \quad \mbox{and} \quad 
\lVert\nabla S^h\rVert_{L^2(\Omega)}^2\leq \frac{C}{h^2}\int_{\Omega} |\nabla R^h|^2\leq C.
\end{equation*}
Passing to a subsequence, we can assume that:
\begin{equation}\label{conv_deba}
S^h \rightharpoonup S \quad \mbox{ weakly in } W^{1,2}(\Omega),
\end{equation}
which implies:
\begin{equation}\label{S_conva}
I_2^h \rightarrow \pare{SQ_0}_{3\times 2} \quad \mbox{in } 
L^2(\Omega, \mathbb{R}^{3\times 2}).
\end{equation}
Consequently, by \eqref{nablaVa}:
\begin{equation}\label{SQa}
\nabla V^h \rightarrow \pare{SQ_0}_{3\times 2} \quad \mbox{in }  L^2(\Omega,
\mathbb{R}^{3\times 2}).
\end{equation}
As before, we conclude that $V^h$ converges in $W^{1,2}(\Omega)$ and that 
its limit $V$ belongs to $W^{2,2}(\Omega, \mathbb{R}^3)$, since
$\nabla V = (SQ_0)_{3\times 2}\in  W^{1,2}(\Omega)$.  
We now prove (\ref{sym0a}). By definition of $S^h$:
\begin{equation}\label{syma}
\sym{S^h}=-\frac{h}{2} (S^h)^tS^h ,
\end{equation}
so in view of the boundedness of $\{S^h\}$ in $W^{1,2}$:
\begin{equation*}
\rVert\sym{S^h}\lVert_{L^2(\Omega)} \leq Ch \rVert S^h\lVert_{L^4(\Omega)}^2
\leq C h \rVert S^h\lVert_{W^{1,2}(\Omega)}^2 \leq C h.
\end{equation*}
Consequently, $S$ is a skew symmetric field. But $(\nabla y_0)^t\nabla V
= (Q_0^t S Q_0)_{2\times 2}$, hence (\ref{sym0a}) follows. 

For future use, let us define $\vec p \in W^{1,2}(\Omega, \mathbb{R}^3)$ by: 
\begin{equation}\label{whatSanew}
[\nabla V ~|~ \vec p] =S Q_0.
\end{equation}
Since $Q_0^t  [\nabla V ~| ~p] \in so(3)$, it is easily checked that $\vec p$ is given solely in terms of $V$ by:
\begin{equation}\label{def_p2}
\begin{cases}
(\nabla y_0)^t \vec p =-(\nabla V)^t \b \\
\inn{\b}{\vec p} =0.
\end{cases}
\end{equation}
%and we have:
%\begin{equation}\label{whatSa}
%S=(Q_0^t)^{-1} \left[\begin{array}{c|c}
%\pare{\nabla y_0}^t \nabla V &  -(\nabla V)^t\b \\ \hline
%(\b)^t\, \nabla V & 0
%\end{array}\right] Q_0^{-1} = [\nabla V ~| ~\vec p]~Q_0^{-1}.
%\end{equation}

\medskip

{\bf 3.}  We now want to establish convergence in (iii). In view of \eqref{nablaVa} we write:
\begin{equation}\label{erra}
\begin{split}
\frac{1}{h}  \sym{ (Q_0^t\nabla V^h) }_{2 \times 2}(x') & = 
%& =\sym{ \pare{ Q_0^t (\bar{R}^h)^t \fint_{-1/2}^{1/2}Z^h(x',x_3)\dx_3}_{2\times 2} } 
\frac{1}{h}  \sym{ (Q_0^t I_1^h)_{2\times 2}}  + \frac{1}{h}\sym{ \pare{Q_0^t S^h Q_0}_{2\times 2} } \\
& = J_1^h+J_2^h.
\end{split}
\end{equation}
We first deal with the sequence $J_2^h$. By
  \eqref{conv_deba}, $S^h\rightarrow S$ in $L^4(\Omega)$  and so
\eqref{syma} implies: 
\begin{equation*}%\label{con}
\frac{1}{h}\sym{S^h} \rightarrow -\frac{1}{2}{ S^tS} = \frac{1}{2}S^2
\quad \mbox{in } L^2(\Omega).
\end{equation*} 
Therefore:
\begin{equation}\label{Sa}
J_2^h \rightarrow -\frac{1}{2}\pare{Q_0^tS^tSQ_0}_{2\times 2}
= -\frac{1}{2}(\nabla V)^t\nabla V \quad \mbox{in} \quad L^2(\Omega).
\end{equation} 
We now prove that $J^h_1$ converges. Recall that by (\ref{erra}), (\ref{nablaVa}) and  (\ref{yu}):
\begin{equation}
 J_1^h = \frac{1}{h}  \sym{  (Q_0^t I_1^h)_{2\times 2}} = 
 \sym{ \pare{ Q_0^t (\bar{R}^h)^t \fint_{-1/2}^{1/2}Z^h(x',x_3)\dx_3}_{2\times 2} } 
\end{equation}
where the rescaled strains $Z^h$ are defined by:
%Consider the following sequence of rescaled
%strains,  bounded in $L^2(\Omega^1, \mathbb{R}^3)$: 
\begin{equation}\label{Zh} 
Z^h(x',x_3)=\frac{1}{h^2} \left(\nabla u^h(x', hx_3)-R^h(x')(Q_0(x')+hx_3B_0(x'))\right).
\end{equation}
By (\ref{estimate3}), the sequence $\{Z^h\}$ is bounded in
$L^2(\Omega^1, \mathbb{R}^3)$. Therefore, up to a subsequence: 
\begin{equation}\label{conv_Za}
Z^h\rightharpoonup Z \quad \mbox{weakly in }   L^2(\Omega^1, \mathbb{R}^3).
\end{equation}
It follows that:
\begin{equation}\label{con2a}
J_1^h \rightharpoonup J_1:=  \sym{ \pare{ Q_0^t (\bar{R})^t \fint_{-1/2}^{1/2}Z(x',x_3)\dx_3}_{2\times 2} } 
\quad \mbox{weakly in }  L^2(\Omega).
\end{equation}
which yields (iii) by  (\ref{erra}) and (\ref{Sa}).

\medskip

{\bf 4.} We now aim at giving the structure of the weak limit
$\mathbb{S}$ of $\frac{1}{h} \sym{(Q_0^t\nabla V^h)}_{2 \times 2}$ in
terms of the limiting fields  $V$ and $Z$. We have just seen that:
\begin{equation}\label{S}
\mathbb{S}= J_1 -\frac{1}{2}(\nabla V)^t\nabla V, 
\end{equation}
where $J_1$ is given by (\ref{con2a}).
As a tool,  consider the difference quotients $f^{s,h}$:
\begin{equation*}
f^{s,h}(x', x_3)= \frac{1}{h^2s} \pare{ y^h(x', x_3+s)-y^h(x',x_3)-hs \pare{\b+h\big(x_3+\frac{s}{2}\big)\d}},
\end{equation*}
and let us study for any $s$ the convergence  of  $f^{s,h}$ as $h\to
0$. In fact, we will show that $f^{s,h} \rightharpoonup \vec p$, 
weakly in $W^{1,2}(\Omega^1,\mathbb{R}^3)$. Write:
\begin{equation*}
f^{s,h}(x', x_3)= \frac{1}{h^2} \fint_0^s\partial_3 y^h(x', x_3+t)-h(\b+h(x_3+t)\d) \,\mathrm{d}t,
\end{equation*}
and observe that:
\begin{equation*}
\begin{split}
\frac{1}{h^2}&\pare{ \partial_3 y^h-h(\b+hx_3\d) }
= \frac{1}{h} \pare{ (\bar{R}^h)^t\nabla u^h(x', hx_3)-(Q_0+hx_3B_0) }e_3 \\
&= \frac{1}{h} (\bar{R}^h)^t \pare{ \nabla u^h(x',hx_3)-R^h(Q_0+hx_3B_0) }e_3  +   S^h(Q_0+hx_3B_0)e_3\\
& = h (\bar{R}^h)^t  Z^h(x',x_3) e_3 +   S^h(Q_0+hx_3B_0)e_3.\\
\end{split}
\end{equation*}
The first term in the right hand side above converges to $0$ in
$L^2(\Omega^1)$ because $\{Z^h\}$ is bounded in $L^2(\Omega^1,
\mathbb{R}^3)$, while the second term converges to
$SQ_0e_3 = S\vec b_0$ in $L^2(\Omega^1)$ by  
\eqref{conv_deba}. Note that $SQ_0e_3=\vec p$ by \eqref{whatSanew}.  
Therefore, $f^{s,h} \rightarrow \p$ in $L^2(\Omega^1)$.

We now deal with the derivatives of the studied sequence. Firstly:
\begin{equation*}
\begin{split}
\partial_3 f^{s,h}(x', x_3) = & \frac{1}{s} \Big(
\frac{1}{h^2}\pare{ \partial_3 y^h(x', x_3+s)-h(\b+h(x_3+s)\d) } \\
& \qquad -\frac{1}{h^2}\pare{ \partial_3 y^h(x', x_3)-h(\b+hx_3\d) } \Big)\\
\end{split}
\end{equation*}
converges to $0$ in $L^2(\Omega^1)$.  For $i=1,2$, the in-plane derivatives read as:
\begin{equation*}
\begin{split}
\partial_i f^{s,h} (x', x_3) &= \frac{1}{h^2s} \Big(
  (\bar{R}^h)^t \partial_i  u^h(x', h(x_3+s)) \\ & \qquad\qquad - (\bar{R^h})^t \partial_i
  u^h(x',hx_3) -hs\pare{ \partial_i\b+h\pare{x_3+\frac{s}{2}}
  }\partial_i\d\Big)\\ 
& =\frac{1}{s}\pare{ (\bar{R}^h)^t Z^h(x', x_3+s)-(\bar{R^h})^t Z^h(x',x_3)  }e_i\\
& \quad + \frac{1}{h^2s}\pare{ (\bar{R}^h)^t R^h(Q_0+h(x_3+s)B_0) -(\bar{R}^h)^t R^h(Q_0+hx_3B_0) }e_i \\
& \quad -\frac{1}{h}\pare{B_0e_i+h\big(x_3+\frac{s}{2}\big)\partial_i \d}.
\end{split}
\end{equation*}
The last two terms above can be written as: $S^hB_0e_i-\pare{x_3+\frac{s}{2}}\partial_i\d$, hence by
\eqref{conv_Za}:
\begin{equation*}
\begin{split}
\partial_i f^{s,h}(x',x_3)\rightharpoonup 
& \frac{1}{s}(\bar{R})^t\Big( Z(x',x_3+s)-Z(x',x_3) \Big) e_i
\\ & \qquad\qquad + SB_0e_i -  \pare{x_3+\frac{s}{2}}\partial_i\d
\quad \mbox{weakly in }  L^2(\Omega^1, \mathbb{R}^3),
\end{split}
\end{equation*}
where $\bar R\in SO(3)$ is an accumulation point of the rotations
$\bar{R}^h$.

Consequently, $f^{s,h} \rightharpoonup \vec p$ weakly in $W^{1,2}(\Omega^1,\mathbb{R}^3)$ and, for $i=1,2$:
\begin{equation}\label{form-a}
s \partial_i \vec p =(\bar{R})^t\Big( Z(x',x_3+s)-Z(x',x_3) \Big) e_i
+ s SB_0e_i -s \pare{x_3+\frac{s}{2}}\partial_i\d, 
\end{equation}
which proves that $Z(x', \cdot) e_i$ has polynomial form and that: 
\begin{equation}\label{forma}
 \pare{ \bar{R}^tZ(x',x_3) }_{3\times 2}
=\pare{ \bar{R}^tZ(x',0) }_{3\times 2}
+ x_3 \pare{\nabla \vec p -  (SB_0 )_{3\times 2}}+ \frac{x_3^2}{2}\nabla\d.
\end{equation}

\medskip

By \eqref{conv_Za}, it follows that:
\begin{equation*}
J_1= \sym{\pare{ Q_0^t(\bar{R})^t Z(x',0) }_{2\times 2}}+\frac{1}{24}\sym{(Q_0^t\nabla\d)_{2\times 2}}.
\end{equation*}
With (\ref{S}),  we finally arrive at the following identity that
links $\mathbb{S}$ and $V$ and $Z$: 
\begin{equation}\label{Ba}
\begin{split}
\mathbb{S}(x') = \sym{\pare{ Q_0^t(\bar{R})^t Z(x',0) }_{2\times 2}}+\frac{1}{24}\sym{(Q_0^t\nabla\d)_{2\times 2}}
-\frac{1}{2}(\nabla V)^t\nabla V.
\end{split}
\end{equation}

\medskip

{\bf 5.} We now prove the lower bound in (iv). Recall that by (\ref{Zh}):
 \begin{equation*}
  \nabla u^h(x', hx_3) = R^h(x')(Q_0(x')+hx_3B_0(x')) + h^2 Z^h(x',x_3).
\end{equation*}
 Since $Q_0A^{-1}\in SO(3)$ we have:
\begin{equation*}
W(\nabla u^h A^{-1}) = W\big((Q_0A^{-1})^t(R^h)^t\nabla u^h A^{-1}\big)
= W\big(\Id +h\mathcal{J} +h^2\mathcal{G}^h\big),
\end{equation*}
where: 
\begin{equation*}
\mathcal{J}(x',x_3)=x_3A^{-1}(Q_0^tB_0)A^{-1} (x')\in so(3), \quad
\mathcal{G}^h(x',x_3) =  A^{-1}Q_0^t(R^h)^tZ^h(x',x_3)A^{-1}. 
\end{equation*} 
Note that by \eqref{conv_Za}: 
\begin{equation*}
\mathcal{G}^h(x',x_3) \rightharpoonup \mathcal{G}  = A^{-1}Q_0^t(\bar{R}^t)Z(x',x_3)A^{-1}
\quad \mbox{weakly in } L^{2}(\Omega^1, \mathbb{R}^{3\times 3}).
\end{equation*} 
Define the ``good sets'': 
$$\Omega_h=\{ x\in\Omega^1; ~ h|\mathcal{G}^h|<1 \}.$$ 
By the above, the characteristic functions $\car{\Omega_h}$ converge to $\car{}$ in $L^1(\Omega)$.
Further, by frame invariance and Taylor expanding $W$ on $\Omega_h$:
\begin{equation*}
\begin{split}
W\big(\Id + h\mathcal{J} + h^2\mathcal{G}^h\big)
&= W\big(e^{-h\mathcal{J}} (\Id + h\mathcal{J} +h^2\mathcal{G}^h)\big) \\
&= W(\Id+h^2(\mathcal{G}^h-\frac{1}{2}\mathcal{J}^2)+\smo{h^2})\\
& = \frac{1}{2}\mathcal{Q}_3\pare{h^2(\mathcal{G}^h-\frac{1}{2}\mathcal{J}^2)}+\smo{h^4}.
\end{split}
\end{equation*}
Therefore:
\begin{equation}\label{Q2a}
\begin{split}
\liminf_{h\to 0}\frac{1}{h^4}E^h(u^h)
& \geq \liminf_{h\to 0} \frac{1}{h^4} \int_{\Omega^1}\car{\Omega_h}
W\big(\Id + h\mathcal{J} +h^2\mathcal{G}^h \big) \dx \\ 
& = \liminf_{h\to
  0}\frac{1}{2}\int_{\Omega^1}\mathcal{Q}_3\Big(\car{\Omega_h}\sym{\Big(
  \mathcal{G}^h -\frac{1}{2}\mathcal{J}^2 }\Big)\Big)\dx \\ 
&\geq \frac{1}{2}\int_{\Omega^1}\mathcal{Q}_3\Big(\sym{\Big(
  \mathcal{G} -\frac{1}{2}\mathcal{J}^2 \Big)}\Big)\dx, 
\end{split}
\end{equation}
by the weak sequential lower semi-continuity of the quadratic form
$\mathcal{Q}_3$ in $L^2$ and in view of:
$$\car{\Omega^h}\sym{\Big(\mathcal{G}^h
  -\frac{1}{2}\mathcal{J}^2\Big)}\rightharpoonup \sym{\mathcal{G}}
-\frac{1}{2}\mathcal{J}^2 \quad \mbox{weakly in } L^2(\Omega^1).$$
Note that by \eqref{whatSanew}  we have:  $\pare{Q_0^tSB_0}_{2\times
  2}=-(\nabla V)^t\nabla \b$ and that:
$$\mathcal{J}^2=-\mathcal{J}^t\mathcal{J}=-x_3^2A^{-1}B_0^tB_0A^{-1}.$$ 
Therefore, using \eqref{forma}, the right hand side of \eqref{Q2a} is
bounded below by:
\begin{equation*}\label{longa}
\begin{split} 
& \frac{1}{2} \int_{\Omega^1} \mathcal{Q}_{2,A} \Big(x', \mbox{sym} 
\Big( Q_0^t(\bar{R})^t Z(x',0)
+ x_3\big(Q_0^t\nabla \vec p + (\nabla V)^t\nabla\b \big)
+
\frac{x_3^2}{2}\big(Q_0^t\nabla\d+(\nabla\b)^t\nabla\b\big)\Big)_{2\times
  2} \Big)~\mbox{d}x\\ & \qquad\qquad = \frac{1}{2}\int_{\Omega^1}
\mathcal{Q}_{2,A}\Big(x', I(x')+x_3III(x')+x_3^2II(x')\Big)\dx.
\end{split}
\end{equation*}
Above we used \eqref{Ba} and we denoted:
\begin{equation}\label{formy2}
\begin{split}
I(x')&=\mathbb{S}-\frac{1}{24}\sym{((\nabla y_0)^t\nabla\d) + \frac{1}{2} (\nabla V)^t\nabla V }\\
II(x')&=\frac{1}{2}\sym{ ((\nabla y_0)^t\nabla \d)+\frac{1}{2}(\nabla\b)^t\nabla\b} \\
III(x')&=\sym ((\nabla y_0)^t\nabla\vec p) +\sym((\nabla V)^t\nabla\b).
\end{split}
\end{equation}
Let $\mathcal{L}_{2, A}(x')$ be the symmetric bilinear form generating
the quadratic form $\mathcal{Q}_{2, A}(x')$.
Since the odd powers of $x_3$ integrate to $0$ on the symmetric
interval $(-1/2, 1/2)$, we get:
\begin{equation*}
\begin{split}
\int_{\Omega^1} & \mathcal{Q}_{2,A}\Big(x',
I(x')+x_3III(x')+x_3^2II(x')\Big)\dx \\ & = 
\int_{\Omega}
\mathcal{Q}_{2,A}(x', I(x'))~\mbox{d}x' + (\int_{-1/2}^{1/2} x_3^2
~\mbox{d}x_3) \int_\Omega Q_{2, A}(x', III(x'))~\mbox{d}x' \\ & \quad +  (\int_{-1/2}^{1/2} x_3^4
~\mbox{d}x_3) \int_\Omega Q_{2, A}(x', II(x'))~\mbox{d}x' +  2 (\int_{-1/2}^{1/2} x_3^2
~\mbox{d}x_3) \int_\Omega \mathcal{L}_{2, A}(x', I(x'),
II(x'))~\mbox{d}x' \\ & = \int_{\Omega}
\mathcal{Q}_{2,A}(x', I) + \frac{1}{12}\int_\Omega
Q_{2, A}(x', III) +  \frac{1}{80}
\int_\Omega Q_{2, A}(x', II) +  \frac{2}{12}\int_\Omega
\mathcal{L}_{2, A}(x', I,  II)~\mbox{d}x' \\ & = \int_{\Omega}
\mathcal{Q}_{2,A}\Big(x', I + \frac{1}{12} II \Big)~\mbox{d}x' + \frac{1}{12}\int_\Omega
Q_{2, A}(x', III)~\mbox{d}x' + \frac{1}{180}
\int_\Omega Q_{2, A}(x', II)~\mbox{d}x' \\ &
= \mathcal{I}_4(V, \mathbb{S}),
\end{split}
\end{equation*}
by a direct calculation. This completes the proof of Theorem
\ref{lowerless4} in view of (\ref{Q2a}).
\end{proof}

\section{The upper bound}

We now complete the proof of $\mathcal{I}_4$ being the $\Gamma$-limit
of $h^{-4}E^h$, by proving that the lower bound (\ref{limit_fun}) is optimal. 

\begin{theorem}\label{upperless4}
Let $V\in W^{2,2}(\Omega, \mathbb{R}^3)$ and $\mathbb{S}\in
L^2(\Omega, \mathbb{R}_{sym}^{2\times 2})$ satisfy:
\begin{equation}\label{finst}
\begin{split}
& \sym \big((\nabla y_0)^t \nabla V\big)=0,\\
& \mathbb{S}\in \mathcal{S} := \mathrm{cl}_{L^2}\big\{ \sym{((\nabla y_0)^t
  \nabla w)}; ~ w \in W^{1,2}(\Omega, \mathbb{R}^3) \big\}.
\end{split}
\end{equation}
Then there exists a sequence $u^h\in W^{1, 2}(\Omega^h, \mathbb{R}^3)$
such that assertions (i), (ii) and (iii) of Theorem \ref{lowerless4}
are satisfied with $R^h=\mathrm{Id}$ and $c^h=0$, and:
\begin{equation}
\limsup_{h\to 0} \frac{1}{h^4} E^h(u^h) \leq \mathcal{I}_4(V,\mathbb{S}).
\end{equation}
\end{theorem}

\begin{proof}
In the construction below, we will use the following notation.
In view of (\ref{Q2A}), for every $F_{2\times 2}\in
\mathbb{R}^{2\times 2}$ one can write:
\begin{equation}\label{Q2min}
\begin{split}
\mathcal{Q}_{2, A}(x', F_{2\times 2})
=\min_{c\in\mathbb{R}^3}  \Big\{ \mathcal{Q}_3\big(A^{-1}( F^*_{2\times 2}+\sym (c\otimes e_3) )A^{-1}\big)  \Big\},
\end{split}
\end{equation}
where $F_{2\times 2}^*$ denotes the $\mathbb{R}^{3\times 3}$ matrix
whose principal $2\times 2$ minor equals $F_{2\times 2}$.
We will denote by $c(x', F_{2\times 2})$ the unique minimizer in 
\eqref{Q2min}.   Note that $c(x', \cdot)$ is a linear function of
$F_{2\times 2}$ and it depends only on its symmetric part $(\sym F_{2\times 2})$.

\medskip

{\bf 1.}  Since $\mathbb{S}\in\mathcal{{S}}$, there exists a
sequence $w^h\in W^{1,2}(\Omega, \mathbb{R}^3)$ such that: 
\begin{equation}\label{conveB}
\sym \big((\nabla y_0)^t\nabla(w^h+\frac{1}{24}\d)\big) \rightarrow \mathbb{S}
\quad \mbox{in } L^2(\Omega, \mathbb{R}^{2\times 2})
\end{equation} 
and without loss of generality we can assume that each $w^h$ is smooth up to
the boundary of $\Omega$, together with:
\begin{equation}\label{stimaw2}
\lim_{h\to 0} \sqrt{h} \lVert w^h \rVert_{W^{2,\infty}} = 0.
\end{equation}
Fix a small $\epsilon_0\in (0, 1)$ and 
let $v^h\in W^{2,\infty}(\Omega, \mathbb{R}^3)$ be a sequence of
Lipschitz deformations with the properties:
\begin{equation}\label{measure}
\begin{split}
v^h&\rightarrow V \quad \mbox{in } W^{2,2}(\Omega, \mathbb{R}^3), \\
h&\lVert v^h\rVert_{W^{2, \infty}} \leq \epsilon_0, \\
\lim_{h\to 0}\frac{1}{h^2} &\big| \big\{ x'\in\Omega; ~ v^h(x')\neq V(x') \big\} \big| = 0.
\end{split}
\end{equation}
We refer to \cite{liu} and \cite{FJMhier} for the construction of such
truncated sequence $v^h$. Define now $\ph\in W^{1,\infty}(\Omega,
\mathbb{R}^3)$ by:
\begin{equation}
\ph= (Q_0^t)^{-1}
\left[\begin{array}{c}
-(\nabla v^h)^t\b \\  0
\end{array}\right],
\end{equation}
and also define the fields  $\qh\in W^{1,\infty}(\Omega, \mathbb{R}^3)$,
$\k$ smooth and $\tilde r^h\in L^\infty(\Omega, \mathbb{R}^3)$ such that:
\begin{equation*}
\begin{split}
& Q_0^t\qh=\frac{1}{2}c\big(x', 2 (\nabla y_0)^t\nabla w^h + (\nabla v^h)^t\nabla v^h\big) -
\left[\begin{array}{c} (\nabla w^h)^t\b \\ 0 \end{array}\right]
-\left[\begin{array}{c}
(\nabla v^h)^t\ph \\ 
\frac{1}{2}|\ph|^2
\end{array}\right],\\ 
& Q_0^t\vec k_0=c\big(x', (\nabla y_0)^t\nabla \d + (\nabla \b)^t\nabla \b\big)
-\left[\begin{array}{c}
(\nabla \b)^t\d \\  |\d|^2 \end{array}\right],\\
& Q_0^t\tilde{r}^h =c\big(x', (\nabla y_0)^t\nabla \ph + (\nabla v^h)^t\nabla \b\big)
-\left[\begin{array}{c} (\nabla v^h)^t\d \\  \langle \ph, \d\rangle \end{array}\right].
\end{split}
\end{equation*}
Finally, let $\rh\in W^{1,\infty}(\Omega, \mathbb{R}^3)$ be such that:
\begin{equation}\label{stimarh}
\begin{split}
 \lim_{h\to 0} \lVert \rh -\tilde{r}^h \rVert_{L^2} = 0,  \qquad 
\lim_{h\to 0} \sqrt{h}~\lVert \rh \rVert_{W^{1,\infty}} = 0.
\end{split}
\end{equation}
It follows from the definition of the minimizing map $c$, that:
\begin{equation}\label{pomoc3}
\begin{split}
&\mathcal{Q}_3\Big(A^{-1} \big( 2 Q_0^t[\nabla w^h \;|\; \qh]  + [\nabla
v^h \;|\;\ph ]^t[\nabla v^h \;|\;\ph]  \big) A^{-1}\Big) \\
&\qquad\qquad\qquad \qquad\qquad\qquad\qquad\qquad 
= \mathcal{Q}_{2,A}\Big(x', 2 (\nabla y_0)^t\nabla w^h + (\nabla v^h)^t\nabla v^h\Big),\\
& \mathcal{Q}_3\Big(A^{-1} \big( Q_0^t[\nabla\d \;|\; \k]  + [\nabla \b \;|\;\d ]^t[\nabla \b \;|\;\d] \big)A^{-1}\Big) \\
& \qquad\qquad\qquad \qquad\qquad\qquad\qquad\qquad 
=\mathcal{Q}_{2,A}\Big(x', (\nabla y_0)^t\nabla \d +(\nabla \b)^t\nabla\b\Big), \\
& \mathcal{Q}_3\Big( A^{-1}\big( 2 Q_0^t[ \nabla \ph \;|\; \tilde{r}^h
] + 2 [\nabla v^h \;|\; \ph]^t[\nabla \b \;|\; \d]^t  \big)A^{-1} \Big)\\
& \qquad\qquad\qquad \qquad\qquad\qquad\qquad\qquad 
=\mathcal{Q}_{2,A}\Big(x', (\nabla y_0)^t\nabla \ph + (\nabla v^h)^t\nabla\b \Big).
\end{split}
\end{equation}
Moreover, we have the following pointwise bounds: 
\begin{equation}\label{stimequanti}
\begin{split}
|\ph|&\leq C|\nabla v^h|, \\
|\nabla \ph|&\leq C(|\nabla v^h|+|\nabla^2 v^h|),\\
|\qh|&\leq C(|\nabla w^h|+|\nabla v^h|^2+|\nabla v^h||\ph|+|\ph|^2) 
   \leq C(|\nabla w^h|+|\nabla v^h|^2), \\
|\nabla \qh|  &\leq C(|\nabla w^h|+|\nabla^2 w^h|+|\nabla^2 v^h||\nabla v^h|+|\nabla v^h|^2).
\end{split}
\end{equation}

\medskip

{\bf 2.} Consider the sequence $u^h\in W^{1,\infty}(\Omega^h, \mathbb{R}^3)$ defined as:
\begin{equation*}
\begin{split}
u^h(x', x_3)&=y_0(x')+hv^h(x')+h^2w^h(x')+x_3\b(x')+\frac{x_3^2}{2}\d(x') \\
&\quad+\frac{x_3^3}{6}\k(x')+hx_3\ph(x')+h^2x_3\qh(x')+\frac{hx_3^2}{2}\rh(x').
\end{split}
\end{equation*}
For every $(x', x_3)\in\Omega^1$ we write:
\begin{equation*}
\nabla u^h(x', hx_3) = Q_0(x') +  Z_1^h(x', x_3) + Z_2^h(x', x_3),
\end{equation*}
where:
\begin{equation*}
\begin{split}
Z^h_1(x', x_3) & = h[\nabla v^h \;|\; \ph] + h^2[ \nabla w^h \;|\;
\qh] + hx_3[ \nabla \b \;|\; \d] + \frac{h^2x_3^2}{2}[\nabla \d \;|\; \k]
+h^2x_3[\nabla\ph \;|\; \rh], \\
Z^h_2(x', x_3) & = \frac{h^3x_3^3}{6}[\nabla\k \;|\; 0]
+ h^3x_3[ \nabla\qh \;|\; 0] + \frac{h^3x_3}{2}[ \nabla \rh\;|\; 0].
\end{split}
\end{equation*}
Since $Q_0A^{-1} \in SO(3)$, we get:
\begin{equation*}
\nabla u^hA^{-1}(x', hx_3) = Q_0A^{-1}\Big(\Id+A^{-1}Q_0^tZ_1^hA^{-1}+A^{-1}Q_0^tZ_2^hA^{-1}\Big)
\end{equation*}
and, in view of \eqref{measure}, \eqref{stimarh} and
\eqref{stimequanti}, there follows for $h$ sufficiently small:
\begin{equation*}
\begin{split}
&\lVert A^{-1}Q_0^tZ_1^hA^{-1}+A^{-1}Q_0^tZ_2^hA^{-1} \rVert_{L^{\infty}}\\
&\qquad\qquad\leq C\Big( h\|\nabla v^h\|_{L^{\infty}} +h\|\ph \|_{L^{\infty}}
+h^2\|\nabla w^h\|_{L^{\infty}}+h^2\|\qh\|_{L^{\infty}} 
+ h\|\nabla\b\|_{L^{\infty}}  + h\|\d\|_{L^{\infty}} \\ &
\qquad\qquad\qquad\qquad  + h^2\|\nabla\d\|_{L^{\infty}} 
+ h^2\|\k\|_{L^{\infty}} + h^2\|\nabla \ph\|_{L^{\infty}}
+ h^2\|\rh\|_{L^{\infty}} +h^3\|  \nabla\k \|_{L^{\infty}} \\ &
\qquad\qquad\qquad\qquad + h^3
\|\nabla \qh \|_{L^{\infty}} + h^3 \| \nabla\rh \|_{L^{\infty}}\Big) 
\leq C\epsilon_0.
\end{split}
\end{equation*}
By the left polar decomposition, there exists a further rotation $R\in SO(3)$ such that:
\begin{equation*}
\begin{split}
R\nabla u^hA^{-1}
& =  \Big((\Id+A^{-1}Q_0^tZ_1^hA^{-1}+A^{-1}Q_0^tZ_2^hA^{-1})^t
(\Id+A^{-1}Q_0^tZ_1^hA^{-1}+A^{-1}Q_0^tZ_2^hA^{-1})  \Big)^{1/2} \\ 
&= \Big( \Id+2A^{-1}\sym(Q_0^tZ_1^h)A^{-1}+A^{-1}(Z_1^h)^tZ_1^hA^{-1}+\bigo{|Z_2^h|}  \Big)^{1/2} \\
& = \Id +A^{-1}\sym(Q_0^tZ_1^h)A^{-1} + \frac{1}{2}
A^{-1}(Z_1^h)^tZ_1^h A^{-1} \\
& \qquad\qquad\qquad\quad\qquad\qquad\qquad 
+ \mathcal{O}\Big(|\sym(Q_0^tZ_1^h)+(Z_1^h)^tZ_1^h |^2\Big)  +\bigo{|Z_2^h|}.
\end{split}
\end{equation*}

\medskip

{\bf 3.} Consider the set:
$$\Omega_h=\big\{ (x', x_3)\in\Omega; ~ v^h(x')=V(x') \big\}.$$ 
Note that on $\Omega_h$ we have: $\ph=\p$ and $Q_0^t[\nabla v^h\;|\;
\ph]\in so(3)$. Using Taylor's expansion, it follows that:
\begin{equation*}
\frac{1}{h^4}\int_{\Omega_h}W\big(\nabla u(x', hx_3) A^{-1}\big)\dx
=  \frac{1}{2h^4}\int_{\Omega_h}\mathcal{Q}_3\Big(A^{-1}
\big( Q_0^tZ_1^h + \frac{1}{2} (Z_1^h)^tZ_1^h \big)A^{-1}\Big)\dx + \mathcal{E}_1^h,
\end{equation*}
where the error term $\mathcal{E}_1^h$ can be estimated by:
\begin{equation*}
|\mathcal{E}_1^h|
\leq \frac{C}{h^4} \int_{\Omega_h} \big|2\sym(Q_0^tZ_1^h) +(Z_1^h)^tZ_1^h\big|^3
+ |Z_2^h|^2 + \big|2\sym(Q_0^tZ_1^h) +(Z_1^h)^tZ_1^h\big||Z_2^h| \dx.
\end{equation*} 
Now on $\Omega_h$ we also have, by \eqref{stimequanti}:
\begin{equation*}
\begin{split}
\big|2\sym(Q_0^tZ_1^h)+(Z_1^h)^tZ_1^h\big| 
& \leq C\Big( h^2|\nabla w^h|+h^2|\nabla v^h|^2+h^2+h^2|\nabla v^h|+h^2|\nabla^2 v^h|+h^2|\rh| \Big) ,\\
|Z_2^h| & \leq Ch^3\big(1+|\nabla \qh|+|\nabla \rh|\big) \\
 & \leq Ch^3\Big( 1+|\nabla w^h|+|\nabla^2 w^h|+|\nabla^2 v^h||\nabla v^h|+|\nabla v^h|^2+|\nabla \rh| \Big),
\end{split}
\end{equation*}
and therefore, in view of \eqref{stimaw2}, \eqref{stimarh},
\eqref{measure}  and $V\in W^{2,2}$:
\begin{equation*}
\begin{split}
&\frac{1}{h^4}\int_{\Omega_h}\big|2\sym(Q_0^tZ_1^h) +(Z_1^h)^tZ_1^h\big|^3\dx \\
&\quad\quad\quad\leq \frac{C}{h^4}\int_{\Omega_h} h^6|\nabla w^h|^3+h^6|\nabla v^h|^6+h^6
+h^6|\nabla v^h|^3+h^6|\nabla^2 v^h|^3+h^6|\rh|^3 \dx \\
&\quad\quad\quad\leq \frac{C}{h^4} \Big(h^2\lVert\nabla w^h\rVert_{L^\infty} (h^2\lVert \nabla w^h\rVert_{L^2})^2
+ h^6\lVert \nabla V\rVert_{L^6}^6 + h^6|\Omega| + h^6\lVert \nabla V\rVert_{L^3}^3\\
&\quad\quad\quad\quad\qquad 
+h^6\lVert \nabla^2 v^h\rVert_{L^\infty}\lVert\nabla^2 V
\rVert_{L^2}^2+(\sqrt{h}\lVert\rh \rVert_{L^\infty})^3h^{9/2} \Big)  
\to 0 \quad \mbox{as } h\to 0.
\end{split}
\end{equation*}
Analogously:
\begin{equation*}
\begin{split}
& \frac{1}{h^4}\int_{\Omega_h}|Z_2^h|^2\dx \leq \frac{C}{h^4} \int_{\Omega_h} h^5+(h\lVert \nabla
v^h\rVert_{L^\infty})^2h^4|\nabla^2 v^h|^2 +h^6|\nabla v^h|^4 \dx  
\to 0 \quad \mbox{as } h\to 0,\\
& \frac{1}{h^4}\int_{{\Omega}_h}\big|2\sym(Q_0^t)Z_1^h+(Z_1^h)^tZ_1^h\big||Z_2^h|\dx\\
&\qquad\quad\leq \frac{C}{h^4} \int_{\Omega_h}  \Big(h^5|\nabla
w^h|^2 + h^5|\nabla^2 w^h|^2 + h^5|\nabla v^h|^2+h^5+h^5|\nabla V|+h^5|\nabla^2 V|+h^5|\rh| \\ 
&\quad\quad\quad\quad\qquad\qquad + h^5|\nabla V|^2|\nabla^2
V|+h^5|\nabla V||\nabla^2 V|^2 \Big) \dx \leq C\epsilon_0.
\end{split}
\end{equation*}
%because after integrating the first 6 terms are $o(1)$ since $w^h$ is
%bounded in $L^2$ and $V\in W^{2,2}$ therefore $\mathcal{E}_1^h\leq
%C\epsilon_0+o(1)$. 
We therefore conclude that:
\begin{equation}\label{blad}
\limsup_{h\to 0} |\mathcal{E}_1^h|\leq C\epsilon_0.
\end{equation}

\medskip

{\bf 4.} Consider now the error due to integrating on the residual subdomain:
\begin{equation*}
\begin{split}
\mathcal{E}_2^h=\frac{1}{h^4}\int_{\Omega^1 \setminus \Omega_h} W\Big(\nabla u^hA^{-1}(x', hx_3)\Big)\dx
\leq \frac{C}{h^4}\int_{\Omega^1 \setminus \Omega_h} \big|2\sym(Q_0^tZ_1^h)+(Z_1^h)^tZ_1^h\big|^2+|Z_2^h|^2 \dx.
\end{split}
\end{equation*}
Observe that, since the matrix field $[\nabla v^h\;|\;\ph]$ is Lipschitz, we have:
\begin{equation*}
\begin{split}
\Big|\sym(Q_0^t[\nabla v^h \;|\;\ph])(x')\Big|
& \leq C \lVert \nabla v^h\rVert_{W^{1, \infty}} \dist\big(x', \{ v^h=V \}\big)\\
& \leq \frac{C\epsilon_0}{h}\dist\big(x', \{ v^h=V \}\big)
\rightarrow 0 \quad \mbox{in } L^{\infty}(\Omega).
\end{split}
\end{equation*}
The last inequality above follows by a standard argument by contradiction.
If there was a sequence $x^h\in\Omega$ such
that $\dist(x^h,\{ v^h=V \} )\geq ch$, this would imply that: $|\{
x'; ~ v^h(x') \neq V(x') \}| \geq \big|\Omega \cap B(x^h, ch)\big| \geq ch^2$, contradicting \eqref{measure}.
Consequently, by \eqref{stimaw2}, \eqref{stimarh}, \eqref{measure}:
\begin{equation*}
\begin{split}
 |\mathcal{E}_2^h| & \leq 
\frac{C}{h^4} \int_{\Omega^1\setminus \Omega_h} h^2\big |\sym(Q_0^t[\nabla v^h\;|\; \ph]) \big |\dx \\
&\quad\quad \quad +\frac{C}{h^4}\int_{\Omega^1\setminus
  \Omega_h} h^4|\nabla w^h|^2+h^4|\nabla v^h|^4+h^4|\nabla^2
v^h|^2+h^4|\rh|^2+h^4 + h^6|\nabla v^h|^4 \dx\\ 
& \leq \frac{C}{h^4}o(h^2)|\Omega^1\setminus \Omega_h|
+\frac{C}{h^4}\sqrt{h}\lVert \nabla w^h
\rVert_{L^\infty}h^{7/2}|\Omega^1\setminus\mathcal{U}^h|^{1/2}\lVert
\nabla w^h \rVert_{L^2} \\ 
&\quad\quad +C|\Omega^1\setminus \mathcal{U}^h|\lVert \nabla v^h\rVert_{L^8}^4
+Ch\lVert \nabla^2 v^h\rVert_{L^\infty}\frac{1}{h}\lVert \nabla^2 v^h
\rVert_{L^2}|\Omega^1\setminus \mathcal{U}^h|^{1/2} 
+ \frac{1}{h}(\sqrt{h}\lVert \rh \rVert_{L^\infty})^2|\Omega^1\setminus \mathcal{U}^h| \\
&\quad\quad  + (h\lVert \nabla^2 v^h\rVert_{L^\infty})^2\lVert \nabla
v^h\rVert_{L^4}^2| \Omega^1\setminus \mathcal{U}^h |^{1/2} \qquad \to 0 \quad
\mbox{as } h\to 0. 
\end{split}
\end{equation*}
Thus:
\begin{equation*}
\limsup_{h\to 0} \frac{1}{h^4}E^h(u^h)
\leq \limsup_{h\to  0}\frac{1}{h^4}\int_{\Omega_h}\frac{1}{2}\mathcal{Q}_3\Big(
A^{-1}\big(\sym(Q_0^tZ_1^h) + \frac{1}{2} (Z_1^h)^tZ_1^h \big)
A^{-1}\Big)~\mbox{d}x+ C\epsilon_0. 
\end{equation*}

Now on $\Omega_h$ we have:
\begin{equation*}
\begin{split}
&2\sym(Q_0^tZ_1^h)+(Z_1^h)^tZ_1^h \\
&\qquad = 2h^2\Big(  \sym(Q_0^t[\nabla w^h\;|\;
\qh])+\frac{x_3^2}{2}\sym(Q_0^t[\nabla \d\;|\; \k])+x_3\sym(
Q_0^t[\nabla \p\;|\; \rh] ) \Big) \\ 
&\qquad\quad+ h^2\Big( [\nabla V\;|\; \p]^t[\nabla V\;|\;
\p]+x_3^2[\nabla \b\;|\; \d]^t[\nabla \b\;|\; \d]+2x_3\sym([\nabla
V\;|\; \p]^t[\nabla \b\;|\; \d]) \Big) + \mathcal{E}^h,
\end{split}
\end{equation*}
where the present error $\mathcal{E}^h$ is estimated by:
\begin{equation}\label{stimaorr}
\begin{split}
|\mathcal{E}^h|
& \leq C\Big( h^3|\nabla V||\nabla w^h|+h^3|\nabla V|+h^3|\nabla V||\nabla \p|+h^3|\nabla V||\rh| \\
& \qquad\qquad + h^4|\nabla w^h|^2+h^3|\nabla w^h|+h^4|\nabla w^h||\nabla \p +h^4|\nabla w^h||\rh|+h^3\\
& \qquad\qquad + h^3|\nabla \p|+h^3|\rh|+h^4+h^4|\nabla\p|+h^4|\rh|+h^4|\nabla \p|^2 +h^4|\rh|^2 \Big) \\
& \leq Ch^2 \Big(  o(1)\sqrt{h}|\nabla V| + \epsilon_0^2|\nabla^2V| + o(1)\sqrt{h} 
+ o(1)\epsilon_0\sqrt{h}  \Big).
\end{split}
\end{equation}
Consequently:
\begin{equation*}
\begin{split}
& \limsup_{h\to 0}\frac{1}{h^4} E^h(u^h) \\
& \qquad \leq \limsup_{h\to 0}\frac{1}{2}  \int_{\Omega_h}
\mathcal{Q}_3\Big( A^{-1}\big(\sym(Q_0^t[\nabla w^h\;|\; \qh]) +
\frac{1}{2} x_3^2\sym(Q_0^t[\nabla \d\;|\; \k]) \\ & \qquad \qquad \qquad\qquad\qquad\qquad\qquad
+ x_3\sym( Q_0^t[\nabla \p\;|\; \rh] )  + \frac{1}{2}
 [\nabla V\;|\; \p]^t[\nabla V\;|\; \p]\\
& \qquad \qquad \qquad\qquad + \frac{1}{2} 
x_3^2[\nabla \b\;|\; \d]^t[\nabla \b\;|\; \d] + x_3\sym([\nabla V\;|\;
\p]^t[\nabla \b\;|\; \d]) \big) A^{-1}\Big)~\mbox{d}x + C \epsilon_0 \\
& \qquad = \limsup_{h\to 0}\frac{1}{2} \int_{\Omega_h}
\mathcal{Q}_{3}\Big (A^{-1} \big( \sym(Q_0^t[\nabla w^h \;|\; \qh]) + \frac{1}{2}
[\nabla V \;|\; \p]^t[\nabla V \;|\; \p] \\ 
&\qquad\qquad\qquad\qquad \qquad\qquad \qquad
+ \frac{1}{2} x_3^2\sym(Q_0^t[\nabla \d \;|\; \k] )
+ \frac{1}{2} x_3^2[\nabla \b \;|\; \d]^t[\nabla \b \;|\; \d]  \big)A^{-1} \Big)\\
& \qquad\qquad\qquad  + \mathcal{Q}_3\Big(
A^{-1}\big(x_3\sym(Q_0^t[\nabla \p \;|\; \rh])
+ x_3\sym ([\nabla V \;|\; \p]^t[\nabla \b \;|\; \d] ) \big) A^{-1} \Big)\dx + C\epsilon_0.
\end{split}
\end{equation*}
Denoting by: $$I_1(x') = \sym((\nabla y_0)^t\nabla w^h) +
\frac{1}{2}(\nabla v^h)^t\nabla v^h, \qquad
I_2(x')=\frac{1}{2}\sym((\nabla y_0)^t\nabla \d)+(\nabla
\b)^t\nabla\b,$$ 
we have:
\begin{equation*}
\begin{split}
\mathcal{Q}_3\Big(&A^{-1}\big(I_1^*(x')+\sym(c(x', I_1(x'))\otimes
e_3) + x_3^2 I_2^*(x') + x_3^2\sym(c(x', I_2(x'))\otimes e_3)\big) A^{-1} \Big)\\ 
& = \mathcal{Q}_3\Big(A^{-1}\big( (I_1(x')+x_3^2I_2(x'))^*+\sym( c(x',
I_1(x')+x_3^2I_2(x'))\otimes e_3 )\big)A^{-1} \Big) \\ 
& = \mathcal{Q}_{2, A} \Big((I_1(x')+x_3^2I_2(x')\Big),
\end{split}
\end{equation*}
where we have used the definition and linearity of the minimizing map $c$.
Recalling the definitions of the curvature forms $I(x')$, $II(x')$ and $III(x')$ in
\eqref{formy2},  observe that $I_2(x')=2II(x')$ and that $\frac{1}{2}I_1$
converges to $I$ in $L^2$ by \eqref{conveB}.   
Hence:
\begin{equation*}
\begin{split}
\limsup_{h\to 0} \frac{1}{h^4}E^h(u^h)
&\leq \frac{1}{2}\int_{\Omega^1} \mathcal{Q}_{2,A}\Big(I(x')+x_3^2II(x')\Big)\dx 
+\frac{1}{2}\int_{\Omega^1}\mathcal{Q}_{2,A}\Big(x_3III(x')\Big)\dx  +C\epsilon_0 \\
&= \mathcal{I}_4(V,\mathbb{S})+ C\epsilon_0.
\end{split}
\end{equation*}
Since $\epsilon_0>0$ was arbitrary, the proof is achieved by a diagonal argument.
\end{proof}

\section{Discussion of the von K\'arm\'an-like functional (\ref{limit_fun})}\label{secexp}

Theorems \ref{lowerless4} and \ref{upperless4} imply, as
usual in the setting of $\Gamma$-convergence, convergence of
almost-minimizers:

\begin{corollary}
If $u^h\in W^{1,2}(\Omega^h, \mathbb{R}^3)$ is a minimizing sequence
to $h^{-4}E^h$, that is:
$$\lim_{h\to 0}\left(\frac{1}{h^4} E^h(u^h) -
  \inf\frac{1}{h^4}E^h\right) = 0,$$
then the appropriate renormalizations $y^h = (\bar R^h)^t u^h(x',
hx_3) - c^h \in
W^{1,2}(\Omega^1,\mathbb{R}^3) $ obey the convergence statements of
Theorem \ref{lowerless4} (i), (ii), (iii). The convergence of
$h^{-1}\mbox{sym} \big((\nabla y_0)^t\nabla V^h\big)$ to $\mathbb{S}$ in (iii) is
strong in $L^2(\Omega)$. Moreover, any limit $(V,\mathbb{S})$ minimizes the
functional $\mathcal{I}_4$. 
\end{corollary}

\begin{proof}
The proof is standard. The only possibly nontrivial part is the strong convergence of the
scaled tangential strains in (iii), which can be deduced as in Theorem 2.5 in
\cite{lemopa1}.
\end{proof}

\medskip

Let us now compare the functional (\ref{limit_fun}) with the
von-K\'arm\'an theory of thin shells that has been derived
in \cite{lemopa1}. Recall that when $S$ is a smooth $2$d surface in
$\mathbb{R}^3$, the $\Gamma$-limit of the scaled elastic energies
$h^{-4}\big( \frac{1}{h}\int_{S^h} W(\nabla u^h)\big)$ on thin shells
$S^h$ with mid-surface $S$, is:
\begin{equation}\label{vKS}
\tilde{\mathcal{I}}_{4,S} (\tilde V, \tilde{\mathbb{S}}) = \frac{1}{2}\int_S
\mathcal{Q}_2\big(\tilde{\mathbb{S}} - \frac{1}{2}(\tilde A^2)_{tan}\big)\mbox{d}y +
\frac{1}{24}\int_S \mathcal{Q}_2\big((\nabla(\tilde A\vec N) - \tilde A\Pi)_{tan}\big)\mbox{d}y.
\end{equation}
Above, $\Pi$ stands for the shape operator of $S$ and $\vec N$ is the
unit normal vector to $S$. The subscript ${tan}$ means 
taking the restriction of a quadratic form (or an operator) to the
tangent space $T_yS$. The arguments of $\tilde{\mathcal{I}}_{4,S}$ are: 

(i) First order infinitesimal isometries $\tilde V$ on $S$. These are
vector fields $\tilde V\in W^{2,2}(S, \mathbb{R}^3)$ with skew
symmetric covariant derivative, so that one may define:
\begin{equation}\label{fii}
\tilde A\in W^{1,2}(S, so(3)) \quad \mbox{with} \quad \tilde A(y) \tau
= \partial_\tau\tilde V(y) \quad \forall y\in S \quad \forall \tau\in T_yS;
\end{equation}

(ii) Finite strains $\tilde{\mathbb{S}}$ on $S$. 
These are tensor fields $\tilde{\mathbb{S}}\in L^2(S, \mathbb{R}^{2\times 2}_{sym})$ such that:
\begin{equation}\label{as}
\tilde{\mathbb{S}} = L^2 - \lim_{h\to 0} \mbox{sym} (\nabla \tilde w_h)_{tan} 
\quad \mbox{ for some} \quad  \tilde w_h\in W^{1,2}(S, \mathbb{R}^3).
\end{equation}

In the present setting, denote $S=y_0(\Omega)$ and observe that the 1-1
correspondence between $\tilde V$ in (\ref{fii}) and $V$ in
(\ref{sym0a}) is given by the change of variables $V=\tilde V\circ y_0$.
The skew-symmetric tensor field $\tilde A$ on $T_yS$ is then uniquely
given by:
\begin{equation}\label{fiii}
\tilde A(y_0(x')) \partial_ey_0 = \partial_eV(x') \quad \mbox{and}
\quad \tilde
A\vec b_0 = \vec p\qquad \forall e\in\mathbb{R}^2,
\end{equation}
and the finite strains in (\ref{as}) are related to (\ref{finst}) by:
$$\langle \tilde{\mathbb{S}}(y_0(x'))\partial_ey_0, \partial_ey_0\rangle =
\langle{\mathbb{S}}(x')e, e\rangle \qquad \forall e\in\mathbb{R}^2.$$

\medskip

Recall that the first of the two terms in the functional (\ref{vKS}) measures the difference of order $h^2$,
between the (Euclidean) metric on $S$ and the metric of the deformed
surface. Indeed, the amount of stretching of $S$ in the direction $\tau\in T_yS$,
induced by the deformation $u_h = id + h\tilde V + h^2\tilde w$, has
the expansion:
\begin{equation*}
|\partial_\tau u_h|^2 - |\tau|^2 =
h^2\Big(2\langle \partial_\tau\tilde w, \tau\rangle +
|\partial_\tau\tilde V|^2\Big) + \mathcal{O}(h^3) = 
2h^2\Big(\langle (\mbox{sym}\nabla \tilde w)\tau, \tau\rangle -
\frac{1}{2} \langle \tilde{A}^2\tau, \tau\rangle\Big) + \mathcal{O}(h^3).
\end{equation*}
The leading order quantity in the right hand side above coincides with:
$$\langle (\mbox{sym}\nabla w) e,e\rangle
+\frac{1}{2}\langle\partial_eV, \partial_eV\rangle =
\Big\langle\big(\mbox{sym}\nabla w + \frac{1}{2}(\nabla V)^t\nabla V \big)e, e\Big\rangle,$$
where we write $\tau = \partial_e y_0$, for any
$e\in\mathbb{R}^2$. This is precisely the argument of the first term in 
$\mathcal{I}_4(V, \mathbb{S})$, modulo the correction $(\nabla\vec b_0)^t\nabla
\vec b_0$ (equal to the third fundamental form on $S$ in case $\vec
b_0 = \vec N$), due to the incompatibility of the ambient Euclidean metric of
$S^h$ with the given prestrain $G$ on $\Omega^h$.

\medskip

The second term in (\ref{vKS}) measures the difference of order $h$,
between the shape operator $\Pi$ on $S$ and the shape operator $\Pi^h$
on the deformed surface $(id+h\tilde V)(S)$ whose unit normal we
denote by $\vec
N^h$. The amount of bending of $S$, in the direction $\tau\in T_yS$,
induced by the deformation $u_h=id + h\tilde V$ can be estimated by \cite{lemopa1}:
\begin{equation*}
\begin{split}
 (\mbox{Id} +h\tilde A)^{-1} \Pi^h & (\mbox{Id}+h\tilde A)\tau - \Pi\tau
= (\mbox{Id} +h\tilde A)^{-1} \big(\partial_\tau\vec N^h +
\mathcal{O}(h^2)\big)\tau - \Pi\tau \\ & = 
(\mbox{Id} +h\tilde A)^{-1} \Big( (\mbox{Id} +h\tilde A)\Pi\tau + h
(\partial_\tau A)\vec N + \mathcal{O}(h^2)\Big) - \Pi\tau \\ & = 
(\mbox{Id} - h\tilde A) h (\partial_\tau\tilde A) \vec N +
\mathcal{O}(h^2) \\ & = h (\partial_\tau\tilde A) \vec N +
\mathcal{O}(h^2) = h\Big(\nabla (\tilde A\vec N) - \tilde A\Pi\Big) + \mathcal{O}(h^2). 
\end{split}
\end{equation*}
The leading order term in this expansion coincides with the term
$(\nabla y_0)^t\nabla \vec p + (\nabla V)^t\nabla \vec b_0$ when $\vec
b_0 = \vec N$, because in view of (\ref{fiii}):
\begin{equation*}
\begin{split}
\langle(\partial_\tau\tilde A) \vec b_0, \tau\rangle  & =
\langle(\partial_e(\tilde A\vec b_0), \partial_ey_0\rangle -
\langle(\tilde A\partial_e \vec b_0, \partial_ey_0\rangle  =
\langle(\partial_e\vec p,  \partial_ey_0\rangle +
\langle(\partial_e \vec b_0, \tilde A\partial_ey_0\rangle \\ & =
\big\langle(\nabla y_0)^t\nabla\vec p ~e, e\big\rangle -
\big\langle(\nabla V)^t\nabla \vec b_0~ e, e\big\rangle,
\end{split}
\end{equation*}
where we again wrote $\tau = \partial_e y_0\in
T_{y_0(x')} S$, for any $e\in\mathbb{R}^2$. This is precisely the
argument in the second term in $\mathcal{I}_4(V,\mathbb{S})$.

In the next section we identify the geometric significance of the last
term in (\ref{limit_fun}).

\section{The scaling optimality}

In this section, we prove the following crucial result:

\begin{theorem}\label{th61}
Assume (\ref{vanish}), together with:
\begin{equation}\label{nowe}
\sym \big((\nabla y_0)^t\nabla \vec d_0\big)  + (\nabla\vec b_0)^t \nabla\vec
b_0 =0,
\end{equation}
where $y_0$, $\b$ and $\d$ are defined in (\ref{sym0}), (\ref{Q_0}), (\ref{d}).
Then the metric $G$ is flat, i.e. $Riem(G)\equiv 0$ in
$\Omega^h$. Equivalently: $\min E^h=0$ for all $h$.
\end{theorem}

Observe that when $\vec b_0 = \vec N$, then by
(\ref{d}) there must be $\vec d_0 = 0$, and hence condition
(\ref{nowe}) becomes: $\vec N\equiv const$. This is consistent with our
previous observation that when $Ge_3 = e_3$, then already condition
(\ref{sym0}) is enough to conclude immersability of $G$ in
$\mathbb{R}^3$. Equivalently, $G_{2\times 2}$ is immersible in
$\mathbb{R}^2$, so that indeed $y_0(\Omega)$ must be planar in this case.

\medskip

Towards a proof of Theorem \ref{th61}, recall that $Riem(G)$ is the covariant Riemann curvature tensor, whose
components $R_{....}$ and their relation to the contravariant curvatures in $R^._{...}$ are:
\begin{equation*}
\begin{split}
& R_{iklm} = \frac{1}{2}\big( \partial_{kl}G_{im} + \partial_{im}G_{kl}
  - \partial_{km}G_{il} - \partial_{il}G_{km} \big)
+ G_{np}\big(\Gamma_{kl}^n\Gamma_{im}^p - \Gamma_{km}^n\Gamma_{il}^p\big)\\
& R_{iklm} = G_{is} R^s_{klm},
\end{split}
\end{equation*}
where we used the Einstein summation convention and the Christoffel symbols:
\begin{equation}\label{christ}
\Gamma_{kl}^n = \frac{1}{2}G^{ns} \big(\partial_kG_{sl} + \partial_l G_{sk} -\partial_sG_{kl}\big).
\end{equation}
In view of the symmetries in $Riem(G)$ of a $3$-dimensional metric $G$, its flatness is equivalent to the
vanishing of the following curvatures:
$$R_{1212}, ~~ R_{1213}, ~~ R_{1223}, ~~ R_{1313}, ~~ R_{1323}, ~~ R_{2323}.$$
The proof of Theorem \ref{th61} is a consequence of the following 
observation. 

\begin{theorem}\label{strangeresult}
Assume (\ref{vanish}) and let $y_0$, $\b$ and $\d$ be defined as in (\ref{sym0}), (\ref{d}). Then:
\begin{equation}\label{nowe2}
\sym \big((\nabla y_0)^t\nabla \vec d_0\big)  + (\nabla\vec b_0)^t \nabla\vec
b_0 = \left[\begin{array}{cc} R_{1313} & R_{1323} \\ R_{1323} & R_{2323}\end{array}\right].
\end{equation}
\end{theorem}
\begin{proof}
{\bf 1.} We have:
\begin{equation*}
\begin{split}
& R_{1313} = -\frac{1}{2}\partial_{11}G_{33} +
G_{np}\big(\Gamma^n_{13}\Gamma^p_{13} - \Gamma^n_{11}\Gamma^p_{33}\big), \\
& R_{2323} = -\frac{1}{2}\partial_{22}G_{33} +
G_{np}\big(\Gamma^n_{23}\Gamma^p_{23} - \Gamma^n_{22}\Gamma^p_{33}\big), \\
& R_{1323} = -\frac{1}{2}\partial_{12}G_{33} +
G_{np}\big(\Gamma^n_{13}\Gamma^p_{23} - \Gamma^n_{12}\Gamma^p_{33}\big).
\end{split}
\end{equation*}
On the other hand, in view of (\ref{d}): 
\begin{equation*}
\begin{split}
\forall i,j=1,2 \qquad  \frac{1}{2} \Big( \langle\partial_i y_0, \partial_j \vec d_0\rangle
+ \langle\partial_j y_0, \partial_i \vec d_0\rangle\Big)  & =
\frac{1}{2}\Big(\partial_j\langle\partial_i y_0, \vec d_0\rangle  
+ \partial_i\langle\partial_j y_0, \vec d_0\rangle   \Big)
- \langle\partial_{ij} y_0, \vec d_0\rangle \\ &  = 
- \frac{1}{2}\partial_{ij}G_{33} - \langle\partial_{ij} y_0, \vec d_0\rangle 
\end{split}
\end{equation*}
because: $\partial_j\langle\partial_i y_0, \vec  d_0\rangle  
+ \partial_i\langle\partial_j y_0, \vec d_0\rangle
= - \partial_{ij}|\vec b_0|^2 = - \partial_{ij}G_{33}$.
Consequently, the formula (\ref{nowe2}) will follow, if we establish:
\begin{equation}\label{11-33}
\forall i,j=1,2 \qquad \langle\partial_{ij} y_0, \vec d_0\rangle = G_{np} \Gamma^n_{ij}\Gamma^p_{33}
\quad \mbox{ and } \quad
\langle\partial_i\vec b_0, \partial_j\vec b_0 \rangle = G_{np}\Gamma^n_{i3}\Gamma^p_{j3}.
\end{equation}

\smallskip

{\bf 2.} Before proving (\ref{11-33}) we gather some useful formulas.
Note that $\partial_iG=2\sym{((\partial_iQ)^tQ})$ for
$i=1,2$. Therefore, by direct inspection:
\begin{equation}\label{2part}
\forall i,j,k=1,2\qquad 
\inn{\partial_{ij}y_0}{\partial_k y_0}=\frac{1}{2}(\partial_iG_{kj}+\partial_jG_{ki}-\partial_kG_{ij}).
\end{equation}
Also, recall that condition \eqref{sym0} is equivalent to (see
\cite{BLS}, proof of Theorem 5.3, formula (5.8)):
\begin{equation}\label{bdey}
\forall i,j=1,2\qquad \inn{\partial_{ij}y_0}{\b}=\frac{1}{2}(\partial_iG_{j3}+\partial_jG_{i3}).
\end{equation}
Therefore, for all $i,j=1,2$:
\begin{equation}\label{innyb}
\begin{split}
&\inn{\partial_j y_0}{\partial_i\b}=\partial_i\inn{\partial_j y_0}{\b}-\inn{\partial_{ij}y_0}{\b}
=\frac{1}{2}(\partial_iG_{j3}-\partial_jG_{i3}),\\
&\langle\partial_i \b, \b\rangle= \frac{1}{2}\partial_i G_{33}.
\end{split}
\end{equation}

We now express $\partial_{ij}y_0$, $\partial_i\b$ and $\d$ in the
basis $\{\partial_1y_0, \partial_2y_0, \b\}$, writing:  
\begin{equation}\label{fofo}
\begin{split}
&\partial_{ij}y_0=\alpha^1_{ij}\partial_1y_0+\alpha^2_{ij}\partial_2y_0+\alpha^3_{ij}\b ,\\
&\partial_i\b=\beta_i^1\partial_1y_0 + \beta_i^2\partial_2y_0 +\beta_i^3 \b,\\
&\d=\gamma^1\partial_1 y_0+\gamma^2 \partial_2 y_0+\gamma^3 \b.
\end{split}
\end{equation}
By \eqref{2part}, (\ref{bdey}), \eqref{innyb} and \eqref{d}, it
follows that:
\begin{equation*}
\begin{split}
& G\Big(\alpha^1_{ij},\alpha^2_{ij}, \alpha^3_{ij}\Big)^t = G
Q_0^{-1} \partial_{ij} y_0 =  Q_0^t\partial_{ij}y_0\\ & \qquad\qquad\quad\quad\;\;\;
= \frac{1}{2}\Big(\partial_iG_{1j}+\partial_jG_{1i}-\partial_1G_{ij}, \partial_iG_{2j}+\partial_jG_{2i}-\partial_2G_{ij}, 
\partial_iG_{3j}+\partial_jG_{3i}\Big),\\ 
& G\Big(\beta_i^1, \beta_i^2, \beta_i^3\Big)^t = G Q_0^{-1}\partial_i\vec b_0
= Q_0^t\partial\vec b_0 =  Q_0^t\partial_i\b \\ &
\qquad\qquad\quad\quad =
\frac{1}{2}\Big(\partial_iG_{13}-\partial_1G_{i3}, \partial_iG_{23}-\partial_2G_{i3}, \partial_iG_{33}\Big)^t,\\
& G \big(\gamma^1, \gamma^2, \gamma^3\big)^t = GQ_0^{-1}\vec d_0 =
Q_0^t\vec d_0 = -\frac{1}{2}  \Big( \partial_1 G_{33},  \partial_2 G_{33}, 0\Big)^t.
\end{split}
\end{equation*}
In view of  (\ref{christ}) we then obtain, for all $i,j=1,2$:
\begin{equation*}
\begin{split}
& (\alpha^1_{ij},\alpha^2_{ij}, \alpha^3_{ij}) = (\Gamma^1_{ij},
\Gamma^2_{ij}, \Gamma^3_{ij}), \quad
(\beta_i^1, \beta_i^2, \beta_i^3) = (\Gamma^1_{i3}, \Gamma^2_{i3},
\Gamma^3_{i3}), \quad
 (\gamma^1, \gamma^2, \gamma^3)^t = (\Gamma^1_{33}, \Gamma^2_{33}, \Gamma^3_{33}),
\end{split}
\end{equation*}
so that (\ref{fofo}) becomes:
\begin{equation}\label{exprd}
\begin{split}
&\partial_{ij}y_0=\Gamma^1_{ij}\partial_1y_0+\Gamma^2_{ij}\partial_2y_0+\Gamma^3_{ij}\b, \\
&\partial_i\b= \Gamma^1_{i3}\partial_1y_0+ \Gamma^2_{i3}\partial_2y_0+ \Gamma^3_{i3}\b,\\
&\d=\Gamma^1_{33}\partial_1y_0+\Gamma^2_{33}\partial_2y_0+\Gamma^3_{33}\b.
\end{split}
\end{equation}

\medskip

{\bf 3.} We now prove (\ref{11-33}). Keeping in mind that $Q_0^T Q_0 =
G$, the scalar products of expressions in (\ref{exprd}) are:
\begin{equation*}
\begin{split}
& \inn{\partial_{ij}y_0}{\d} = \big\langle
\Gamma^n_{ij}\partial_ny_0+\Gamma^3_{ij}\b, \Gamma^p_{33}\partial_py_0
  +\Gamma^3_{33}\b\big\rangle = G_{np}\Gamma^n_{ij}\Gamma^p_{33}, \\
& \inn{\partial_{i}\b}{\partial_j\b}=\big\langle\Gamma^n_{i3}\partial_ny_0+\Gamma^3_{i3}\b,
\Gamma^p_{j3}\partial_py_0
  +\Gamma^3_{j3}\b\big\rangle = G_{np}\Gamma^n_{i3}\Gamma^p_{j3}.
\end{split}
\end{equation*}
exactly as claimed in (\ref{11-33}). This ends the proof of Theorem
\ref{strangeresult} and also of Theorem \ref{th61}.
\end{proof}

\section{Two examples}

In this section we compute the energy $\mathcal{I}_4(V, \mathbb{S})$ in the two
particular cases of interest:
\begin{equation*}
G(x', x_3) = \mbox{diag}(1,1,\lambda(x')) \qquad \mbox{and} \qquad
G(x', x_3) = \lambda(x')\mbox{Id}_3.
\end{equation*}
Let $\p$ be as in the definition \eqref{def_p}.  
Writing: $\p=\alpha^1 \partial_1 y_0 + \alpha^2 \partial_2 y_0 + \alpha^3\b$, we obtain:
\begin{equation*}
G \big( \alpha^1, \alpha^2,  \alpha^3\big)^t = - \big( \inn{\partial_1 V}{\b}, \inn{\partial_2 V}{\b}, 0\big)^t.
\end{equation*}
Consequently:
\begin{equation}\label{exprp}
\begin{split}
\p = - G^{1i}\inn{\partial_i V}{\b} \partial_1 y_0 - G^{2i}\inn{\partial_i V}{\b}\partial_2y_0  
- G^{3i}\inn{\partial_i V}{\b}\b.
\end{split}
\end{equation}

\begin{lemma}\label{ex1}
Let $\lambda:\bar{\Omega}\rightarrow \mathbb{R}$ be smooth and
strictly positive. Consider the metric of the form:
$G(x',x_3)=\mathrm{diag}(1,1,\lambda(x'))$. Then:
\begin{itemize}
\item[(i)]  $G$ is immersible in $\mathbb{R}^3$ if and only if:
\begin{equation*}
M_\lambda =\nabla^2\lambda-\frac{1}{2\lambda}\nabla \lambda \otimes
\nabla \lambda \equiv 0 \quad \mbox{in } \Omega,
\end{equation*}
while the condition $M_\lambda\not\equiv 0$ is equivalent to:
$ch^4\leq \inf E^h \leq Ch^4$.
\item[(ii)] The $\Gamma$-limit energy functional $\mathcal{I}_4$ in (\ref{limit_fun}) becomes:
\begin{equation*}
\begin{split}
&  \forall w\in W^{1,2}(\Omega,\mathbb{R}^2) \quad\forall v\in
  W^{2,2}(\Omega,\mathbb{R}) \\ & \qquad \mathcal{I}_4(v,w) = 
\frac{1}{2}\int_{\Omega}\mathcal{Q}_{2}\big(\mathrm{sym} \nabla
w + \frac{1}{2}\nabla v\otimes \nabla v +
\frac{1}{96\lambda}\nabla\lambda\otimes\nabla\lambda\big)~\mathrm{d}x' \\ &
\qquad\qquad\qquad\quad
+\frac{1}{24}\int_{\Omega}\mathcal{Q}_{2}\big(\sqrt{\lambda}\nabla^2 v\big)
+\frac{1}{5760}\int_{\Omega}\mathcal{Q}_{2}\big(M_\lambda\big) ~\mathrm{d}x',
\end{split}
\end{equation*}
where $\mathcal{Q}_2$ is independent of $x'$ and it is defined by $\mathcal{Q}_{2,Id}$ in (\ref{Q2A}).
\end{itemize}
\end{lemma}
\begin{proof}
Part (i) of the assertion has been shown in \cite{BLS}. For (ii), note first that:
$$y_0(x') = x'\quad \mbox{ and } \quad Q_0 = A = \mbox{diag}(1,1,\sqrt{\lambda}).$$
Consequently, directly from (\ref{Q2A}) we see that
$\mathcal{Q}_{2,A}= \mathcal{Q}_{2,Id}$, which we denote simply by
$\mathcal{Q}_2$. 

Further, in view of (\ref{finst}), every admissible
limiting strain $\mathbb{S}\in\mathcal{\mathbb{S}}$ has the form $\mathbb{S}=\mbox{sym}\nabla w$
for some $w\in W^{1,2}(\Omega,\mathbb{R}^2)$. Also, without loss of
generality, every admissible limiting displacement $V$ is of the form:
$V=(0,0,v)$ for some $v\in W^{2,2}(\Omega,\mathbb{R})$. We now
compute, using \eqref{exprb}, \eqref{exprd} and \eqref{exprp}:
\begin{equation*}
\b=\sqrt{\lambda}e_3, \qquad 
\quad \d=-\frac{1}{2}(\partial_1 \lambda, \partial_2 \lambda, 0), \qquad
\quad \p=-\sqrt{\lambda}(\partial_1 v, \partial_2 v,0).
\end{equation*}
Therefore:
\begin{equation*}
\begin{split}
& (\nabla\b)^t\nabla \b  =\frac{1}{4\lambda}\nabla\lambda\otimes\nabla\lambda,\qquad
(\nabla y_0)^t\nabla  \d = -\frac{1}{2}\nabla^2\lambda, \\
& (\nabla y_0)^t\nabla\p = -\frac{1}{2\sqrt{\lambda}}\nabla v\otimes
\nabla \lambda-\sqrt{\lambda}\nabla^2v, \qquad
(\nabla V)^t\nabla\b =\frac{1}{2\sqrt{\lambda}}\nabla v\otimes \nabla\lambda.
\end{split}
\end{equation*}
This ends the proof of Lemma \ref{ex1} in view of \eqref{limit_fun}.
\end{proof}

\begin{lemma}\label{ex2}
Let $\lambda:\bar\Omega\to\mathbb{R}$ be smooth and strictly positive. Consider the metric 
$G(x',x_3)=\lambda(x')\mathrm{Id}_3$. Denote $f=\frac{1}{2}\log \lambda$. Then:
\begin{itemize}
\item[(i)] Condition \eqref{sym0} is equivalent to $\Delta f=0$, which
  is also equivalent to the immersability of the metric $G_{2\times 2}$ in $\mathbb{R}^2$.
\item[(ii)] Under condition \eqref{sym0}, condition \eqref{nowe} can
  be directly seen as equivalent to $\mathrm{Ric}(G)=0$ and therefore to the immersability of $G$.
\item[(iii)] The $\Gamma$-limit energy functional in (\ref{limit_fun}) has the following form:
\begin{equation*}
\begin{split}
\mathcal{I}_4(V,\mathbb{S})
= & ~ \frac{1}{2}
\int_{\Omega}e^{-2f}\mathcal{Q}_2\big(\mathbb{S}+\frac{1}{2}(\nabla V)^t\nabla
V+\frac{1}{24}e^{2f}\nabla f\otimes\nabla f\big) ~\mathrm{d}x' \\ 
& \quad +\frac{1}{24}\int_{\Omega}\mathcal{Q}_2\big(2\nabla V_3 \otimes
\nabla f-\nabla^2 V_3-\inn{\nabla V_3}{\nabla
  f}\mathrm{Id}_2\big) ~\mathrm{d}x'\\  &
\quad +\frac{1}{1440}\int_{\Omega}\mathcal{Q}_2\big(e^{f}\mathrm{Ric}(G)_{2\times 2}\big) ~\mathrm{d}x',
\end{split}
\end{equation*}
where $\mathcal{Q}_2$ is as in Lemma \ref{ex1}, and where
$\mathrm{Ric}(G)_{2\times 2}$ denotes the tangential part of the 
Ricci curvature tensor of $G$, i.e.:
$$\mathrm{Ric}(G)_{2\times 2} = \left[\begin{array}{cc} R_{11} &
    R_{12}\\ R_{12} & R_{22}\end{array}\right].$$
\end{itemize}
\end{lemma}
\begin{proof}
The part $(i)$ has been deduced in \cite{BLS}, together with the expression:
\begin{equation}\label{ricci}
\mathrm{Ric}(G)= -(\nabla^2f-\nabla f \otimes \nabla f)^*-(\Delta f+|\nabla f|^2)\Id.
\end{equation} 
We now consider the case when \eqref{sym0} holds.
By $(i)$ the metric $G_{2\times 2}$ is immersible in $\mathbb{R}^2$
and in particular $\vec N=e_3$. Writing $V=(V_1, V_2, V_3)$, 
from \eqref{exprb}, \eqref{exprd} and \eqref{exprp} we obtain:
\begin{equation*}
\b = \sqrt{\lambda}e_3, \qquad
\d=-\big(\partial_1 f\partial_1 y_0+ \partial_2 f\partial_2 y_0\big),
\qquad \p=-\frac{1}{\sqrt{\lambda}}\big(\partial_1 V_3\partial_1 y_0 +\partial_2 V_3\partial_2 y_0\big).
\end{equation*}
\begin{equation*}
\begin{split} 
(\nabla \vec b_0)^t\nabla \vec b_0 = e^{2f}\nabla f\otimes\nabla f,\qquad
(\nabla V)^t\nabla\b = e^f\nabla V_3\otimes \nabla f.
\end{split}
\end{equation*}
Further, observe that:
$\partial_i \d=-(\partial_{1i}f\partial_1y_0+\partial_{2i}f\partial_2
y_0+\partial_1 f\partial_{1i}y_0+\partial_2f\partial_{2i}y_0 )$, and so:
$$ \frac{1}{\lambda}\inn{\partial_1y_0}{\partial_1\d}
= -\frac{1}{\lambda}\big(\lambda\partial_{11}f
+\frac{1}{2}\partial_1\lambda\partial_1
f+\frac{1}{2}\partial_2\lambda\partial_2 f\big) 
= -(\partial_{11}f+|\nabla f|^2). $$
In the same manner, we arrive at:
\begin{equation*}
\frac{1}{\lambda}\inn{\partial_2y_0}{\partial_2\d} =
-(\partial_{22}f+|\nabla f|^2), \quad 
\frac{1}{\lambda}\inn{\partial_2y_0}{\partial_1\d} = -\partial_{12}f, \quad
\frac{1}{\lambda}\inn{\partial_1y_0}{\partial_2\d} = -\partial_{21}f.
\end{equation*}
Consequently, $(\nabla y_0)^t\nabla\d$ is already a symmetric matrix, and:
$$ (\nabla y_0)^t\nabla\d=-e^{2f}(\nabla^2f+|\nabla f|^2\mathrm{Id}_2).$$
In parti\-cu\-lar, under condition $\Delta f=0$, the formula
(\ref{ricci}) yields:
\begin{equation*}
\sym{(\nabla y_0)^t\nabla\d}+(\nabla\b)^t\nabla \b=
e^{2f}\mathrm{Ric}(G)_{2\times 2},
\end{equation*}
which we directly see to be equivalent with $\nabla f=0$ and hence
with $\mathrm{Ric}(G) = 0$. This establishes $(ii)$.

We now compute the remaining quantities appearing in the expression of
$\mathcal{I}_4$. Firstly:
\begin{equation*}
\nabla\p= \frac{1}{2\lambda^{3/2}}\nabla y_0(\nabla V_3\otimes \nabla\lambda)
-\frac{1}{\sqrt{\lambda}}\nabla
y_0\nabla^2V_3-\frac{1}{\sqrt{\lambda}} \Big( \partial_1V_3
(\partial_{11}y_0, \partial_{12} y_0) + \partial_2V_3
(\partial_{12}y_0, \partial_{22} y_0)\Big).
\end{equation*}
Using the relations between $\inn{\partial_{ij}y_0}{\partial_k
  y_0}$ and $\partial_lG$ in (\ref{2part}), we obtain:
\begin{equation*}
(\nabla y_0)^t\nabla\p= \frac{1}{2\lambda^{3/2}}G_{2\times
  2}\nabla V_3\otimes \nabla
\lambda-\frac{1}{\sqrt{\lambda}}G_{2\times 2}\nabla^2 V_3 
-\frac{1}{2\sqrt{\lambda}}\left[\begin{array}{c|c}
\inn{\nabla V_3}{\nabla \lambda} & \inn{\nabla V_3}{\nabla\lambda^\perp} \\ \hline
-\inn{\nabla V_3}{\nabla\lambda^\perp} &
\inn{\nabla V_3}{\nabla \lambda}
\end{array}\right],
\end{equation*}
and therefore:
$$\sym (\nabla y_0)^t\nabla \vec p = \sqrt{\lambda}~
\mbox{sym} \big(\nabla V_3\otimes \nabla f\big) -
\sqrt{\lambda}\nabla^2V_3 - \sqrt{\lambda}\inn{\nabla V_3}{\nabla
  \lambda}\mathrm{Id}_2.$$  
In a similar manner, it follows that:
\begin{equation*}
\sym (\nabla y_0)^t\nabla\d = -\lambda\Big(
\nabla^2f +|\nabla f|^2\mathrm{Id}_2\Big).
\end{equation*}

Since $\mathcal{Q}_{2, A}(x') = \lambda^{-1}\mathcal{Q}_2$, the
formula in (\ref{limit_fun}) becomes:
\begin{equation}
\begin{split}
\mathcal{I}_4(V,\mathbb{S}) = & \frac{1}{2} \int_{\Omega} e^{-2f} \mathcal{Q}_2\big(
\mathbb{S}+\frac{1}{2}(\nabla V)^t\nabla V+\frac{1}{24}e^{2f}\nabla
f\otimes\nabla f\big) ~\mathrm{d}x'\\
& + \frac{1}{24}\int_{\Omega}e^{-2f}\mathcal{Q}_2\big(2e^f \nabla V_3
\otimes \nabla f-e^f\nabla^2 V_3-e^f\inn{\nabla V_3}{\nabla
  f}\mathrm{Id}_2\big) ~\mathrm{d}x'\\ 
& + \frac{1}{1440}\int_{\Omega}e^{-2f}\mathcal{Q}_2\big(
e^{2f}\mathrm{Ric}(G)_{2\times 2}\big) ~\mathrm{d}x',
\end{split}
\end{equation}
which implies the result.
\end{proof}

\section{Appendix: a proof of Corollary \ref{lemma_estimate2}}

{\bf 1.}  For every $x'\in\Omega$ denote $D_{x', \delta}=B(x',\delta)\cap \Omega$ and
$B_{x', \delta, h}=D_{x', \delta} \times (-h/2,h/2)$. For short, we write 
$B_{x',2h}= B_{x',2h,h}$ and $B_{x',h}= B_{x',h,h}$.
Apply Lemma \ref{Lemma_approx} to the set $\Vh=B_{x',2h}$ to get  a
rotation $\Rhx\in SO(3)$ such that, with a universal constant $C$:
\begin{equation}\label{app1}
\begin{split}
& \frac{1}{h}\int_{B_{x',2h}} \abs{ \nabu(z)- \Rhx\pare{Q_0(z')+z_3B_0(z')} }^2 \dz
\\ & \qquad\qquad\qquad\qquad\qquad\qquad\qquad\qquad
\leq C\pare{ E^h(u^h, B_{x',2h})+h^3\abs{B_{x',2h}} }.
\end{split}
\end{equation}
Consider a family of mollifiers $\Ex \in C^{\infty}(\Omega,
\mathbb{R})$, parametrized by $x'\in\Omega$:
\begin{equation*}
\int_{\Omega}\Ex =\frac{1}{h},\;\;\;
\norm{\Ex}_{L^\infty(\Omega)} \leq \frac{C}{h^3}, \;\;\;\
\norm{\nabla_{x'}\Ex}_{L^\infty(\Omega)} \leq \frac{C}{h^4} \quad
\mbox{and } (\supp{\Ex})\cap\Omega \subset \Dhx.
\end{equation*}
Define $\tilde{R}^h\in W^{1,2}(\Omega, \mathbb{R}^{3\times 3})$ as:
\begin{equation}
\tilde{R}^h(x')=\int_{\Omega^h} \Exz \nabla u^h(z) \pare{Q_0(z')+z_3B_0(z')}^{-1}\dz.
\end{equation}
We then have: 
\begin{equation}\label{est_ball}
\begin{split}
\frac{1}{h}\int_{\Bhx}
|\nabu(z) - & \tilde{R}^h(z')  \pare{Q_0(z')+z_3B_0(z')}|^2 \dz \\
 &\leq \frac{C}{h} \int_{B_{x',2h}} \abs{\nabu(z)-\Rhx\pare{Q_0(z')+z_3B_0(z')}}^2 \dz\\
&\quad + \frac{C}{h} \int_{\Bhx} |\tilde{R}^h(z')-\Rhx|^2 |Q_0(z')+z_3B_0(z')|^2 \dz \\
&\leq C\pare{E^h(u^h, B_{x',2h})+h^3\abs{B_{x',2h}}}+\frac{C}{h}\int_{\Bhx} |\tilde{R}^h(z')-\Rhx|^2 \dz,
\end{split}
\end{equation}
where we have used \eqref{app1} and $\norm{Q_0(z')+z_3B_0(z')}_{L^\infty}\leq C$.
Now, for every $z'\in \Bhx$ we have:
\begin{equation}\label{Rh}
\begin{split}
|&\tilde{R}^h(z') -\Rhx|^2  = \abs{  \int_{\Omega^h} \Eyz \nabu(y)\pare{Q_0(y')+y_3B_0(y')}^{-1}\dy-\Rhx   }^2\\
&=\abs{ \int_{\Omega^h}\Eyz\pare{\nabu(y)-\Rhx\pare{Q_0(y')+y_3B_0(y')}} \pare{Q_0(z')+y_3B_0(z')}^{-1}\dy }^2 \\
&\leq C\pare{\int_{\Bhz}\Eyz^2 \dy} \pare{ \int_{\Bhz} \abs{ \nabu(y)-\Rhx\pare{Q_0(y')+y_3B_0(y')}  }^2 \dy}\\
&\leq \frac{C}{h^2} \int_{B_{x',2h}} \abs{ \nabu(y)-\Rhx\pare{Q_0(y')+y_3B_0(y')}  }^2 \dy\\
&\leq \frac{C}{h^2} \pare{E^h(u^h, B_{x',2h})+h^3\abs{B_{x',2h}}}.
\end{split}
\end{equation}
In a similar way, in view of $\int_{\Omega^h}
\nabla_{z'}\eta_{z'}(y')~\mbox{d}y = 0$, it follows that:
\begin{equation*}
\begin{split}
|\nabla \tilde{R}^h(z')|^2
&= \pare{ \int_{\Omega^h} \nabla_{z'}\Eyz\nabla u^h(y) \pare{Q_0(y')+y_3B_0(y')}^{-1}\dy}^2 \\
&= \pare{ \int_{B_{x', 2h}} \nabla_{z'}\Eyz \pare{\nabla u^h(y) \pare{Q_0(y')+y_3B_0(y')}^{-1}-\Rhx}\dy}^2 \\
&\leq C \int_{\Omega^h}\abs{\nabla_{z'}\Eyz}^2\dy 
\int_{B_{x', 2h}} \abs{ \nabu(y)-\Rhx\pare{Q_0(y')+y_3B_0(y')}  }^2 \dy \\
&\leq \frac{C}{h^4}  \pare{E^h(u^h,B_{x',2h})+h^3\abs{B_{x',2h}}}.
\end{split}
\end{equation*}
From \eqref{Rh} we obtain:
\begin{equation*}
\begin{split}
\int_{\Bhx} |\tilde{R}^h(z')-\Rhx|^2 \dz
&\leq \frac{C}{h^2} \int_{\Bhx}\pare{ E^h(u^h,B_{x',2h})+h^4|B_{x',2h}|} \dz \\
&\leq Ch\pare{ E^h(u^h,B_{x',2h})+h^3|B_{x',2h}|},
\end{split}
\end{equation*}
and therefore by \eqref{est_ball} we further see that:
\begin{equation}\label{est_final}
\begin{split}
\frac{1}{h}\int_{\Bhx}
|\nabu(z)-&\tilde{R}^h(z')\pare{Q_0(z')+z_3B_0(z')}|^2 \dz \\ &\qquad\qquad\qquad
\leq  C\pare{ E^h(u^h, B_{x',2h})+h^3|B_{x',2h}|}.
\end{split}
\end{equation}

\medskip

{\bf 2.} Covering $\Omega^h$ by a finite family of sets  $\{\Bhx\}$,
such that  the intersection number of the doubled covering $\{B_{x',
  2h}\}$ is independent of $h$, applying \eqref{est_final}
and summing over the covering, it follows that:
\begin{equation*}
\frac{1}{h}\int_{\Omega^h}
|\nabu(z)-\tilde{R}^h(z')\pare{Q_0(z')+z_3B_0(z')}|^2 \dz
\leq  C\pare{ E^h(u^h)+h^4}.
\end{equation*} 

In a similar fashion we obtain:
\begin{equation*}
\begin{split}
\int_{\Dhx} |\nabla \tilde{R}^h(z')|^2\dz
&\leq \frac{C}{h^4} \int_{\Dhx}  \pare{E^h(u^h,B_{x',2h})+h^3|B_{x',2h}|} \dz \\
&\leq \frac{C}{h^2}  \pare{E^h(u^h,B_{x',2h})+h^3|B_{x',2h}|},
\end{split}
\end{equation*}
and by the same covering argument:
\begin{equation*}
\int_{\Omega^h} |\nabla \tilde{R}^h(z')|^2\dz
\leq \frac{C}{h^2}  \pare{E^h(u^h)+h^4}.
\end{equation*}

\medskip

{\bf 3.} Note that, in the above two estimates, we can replace $\tilde{R}^h$ by
$R^h=\proj{\tilde{R}^h} \in W^{1,2}(\Omega, SO(3))$. Firstly, the
projection in question is well defined in view of \eqref{Rh}, since:
\begin{equation*}
\dist^2 \pare{\tilde{R}^h, SO(3)} 
\leq |\tilde{R}^h-\Rhx|
\leq \frac{C}{h^2} \pare{ E^h(u^h)+h^4 },
\end{equation*}
which is small because of the hypothesis $\alpha<2$. Moreover:
\begin{equation*}
\begin{split}
\frac{1}{h}\int_{\Bhx}  |\nabla
  u^h(z)-R^h(z') & \pare{Q_0(z')+z_3B_0(z') }|^2\dz \\ 
\leq & ~\frac{C}{h}\int_{\Bhx} \abs{ \nabla u^h(z)-\tilde{R}^h(z')\pare{Q_0(z')+z_3B_0(z') }  }^2\dz \\
& + \frac{C}{h}\int_{\Bhx} |\tilde{R}^h(z')-R^h(z')|^2|Q_0(z')+z_3B_0(z') |^2\dz \\
\leq & ~ C\pare{ E^h(u^h,B_{x',2h})+h^3|B_{x', 2h}| }
\end{split}
\end{equation*}
because of \eqref{est_final} and \eqref{Rh}. Finally, the previous
covering argument clearly implies \eqref{estimate1}, and $\int_{\Omega}
|\nabla R^h|^2\dz \leq C \int_{\Omega} |\nabla \tilde{R}^h|^2\dz$
yields \eqref{estimate2}.

\end{document}